\documentclass[10pt,a4paper,twoside,reqno]{amsart}

\usepackage[a4paper,
left=1in,right=1in,top=1.2in,bottom=1.2in,marginparwidth=2cm,%
footskip=.25in]{geometry}

\usepackage[textsize=footnotesize]{todonotes}

\usepackage{bm}
\usepackage{latexsym, amsmath, amstext, amssymb, amsfonts, amscd, bm, array, multirow, amsbsy, mathrsfs, stmaryrd}
\usepackage{amsthm}
\usepackage{t1enc}
\usepackage[mathscr]{eucal}
\usepackage{indentfirst}
\usepackage{pb-diagram}
\usepackage{graphicx}
\usepackage{fancyhdr}
\usepackage{fancybox}
\usepackage{color}
\usepackage{tikz-cd}
\usetikzlibrary{arrows}
\usepackage[all]{xy}
\usepackage{hyperref}
\usepackage{tikz}
\usepackage{xparse}
\hypersetup{colorlinks=false,pdfborderstyle={/S/U/W 0}}
\usetikzlibrary{matrix}
\usepackage{upgreek}
\usepackage[utf8]{inputenc}
\usepackage[T1]{fontenc}
\usepackage{bbm}
\usepackage{enumitem}
\setlist[enumerate]{leftmargin=*,noitemsep}
\setlist[itemize]{leftmargin=*,noitemsep}
\usepackage{float}
\usepackage{soul}

\newcommand{\doi}[1]{\url{https://doi.org/#1}}
\usepackage[noadjust]{cite}

\makeatletter
\def\@defaultbiblabelstyle#1{[#1]}
\makeatother

\usepackage{url}
\theoremstyle{plain}
\newtheorem{thm}{Theorem}[section]

\newtheorem{prop}[thm]{Proposition}

\newtheorem{lemma}[thm]{Lemma}
\newtheorem{cor}[thm]{Corollary}

\theoremstyle{definition}
\newtheorem{definition}[thm]{Definition}

\theoremstyle{remark}
\newtheorem{remark}[thm]{Remark}
\newtheorem{example}[thm]{Example}

\newtheorem*{ack}{Acknowledgements}

\newcommand{\End}{\mathrm{End}}

\DeclareFontFamily{U}{rsf}{}
\DeclareFontShape{U}{rsf}{m}{n}{<5> <6> rsfs5 <7> <8> <9> rsfs7 <10-> rsfs10}{}
\DeclareMathAlphabet\Scr{U}{rsf}{m}{n}

\def\Z{\mathbb{Z}}
\def\C{\mathbb{C}}
\def\R{\mathbb{R}}

\def\deg{{\rm deg}}

\def\dd{\mathrm{d}}

\def\Id{\mathrm{Id}}

\newcommand{\be}{\begin{equation*}}
\newcommand{\ee}{\end{equation*}}
\newcommand{\ben}{\begin{equation}}
\newcommand{\een}{\end{equation}}
\newcommand{\beqa}{\begin{eqnarray*}}
\newcommand{\eeqa}{\end{eqnarray*}}
\newcommand{\beqan}{\begin{eqnarray}}
\newcommand{\eeqan}{\end{eqnarray}}

\def\scB{\Scr B}

\def\U{\mathrm{U}}

\def\cD{\mathcal{D}}

\def\G_2{\mathrm{G_2}}

\def\cO{\mathcal{O}}

\def\G{\mathrm{G}}

\def\R{\mathbb{R}}

\def\cL{\mathcal{L}}

\def\dd{\mathrm{d}}
\newcommand{\abs}[1]{\left\lvert#1\right\rvert}

\newcommand{\escal}[1]{\langle#1\rangle}

\renewcommand{\ker}{\mathrm{Ker}}

\newcolumntype{P}[1]{>{\centering\arraybackslash}p{#1}}

\usepackage[math]{anttor}
\usepackage{orcidlink}

\allowdisplaybreaks
\usepackage{mlmodern}

\begin{document}

\title[Complex generalized spin-c Killing spinors in low dimensions]{Complex generalized spin-c Killing spinors in low dimensions}

\author[J.\ L.\ Carmona Jiménez]{José Luis Carmona Jiménez \orcidlink{0009-0002-3234-1718}}
\address{Riemann Center for Geometry and Physics, Leibniz Universität Hannover, Hannover, Germany.}
\email{jose.carmona@math.uni-hannover.de}

\author[A.\ Gil-García]{Alejandro Gil-García \orcidlink{0000-0002-9370-241X}}
\address{Scuola Internazionale Superiore di Studi Avanzati, Trieste, Italy}
\email{agilgarc@sissa.it}

\author[C.\ S.\ Shahbazi]{C.\ S.\ Shahbazi \orcidlink{0000-0003-1185-9569}}
\address{Departamento de Matem\'aticas, Universidad UNED - Madrid, Reino de Espa\~na}
\email{cshahbazi@mat.uned.es}

\begin{abstract}
We classify all complex generalized spin-c Killing spinors with complex Killing function $i\mu-\alpha$ on complete Riemannian surfaces, presenting an explicit list of all possible cases. We show that, contrary to previous statements in the literature, there exist complex generalized spin-c Killing spinors on Riemannian surfaces for which $\alpha\mu\not\equiv0$. We illustrate the classification by providing several explicit local, compact, and complete examples, including a complete example for which $\alpha\mu>0$ everywhere. We then move to three dimensions, proving that a spin-c Riemannian three-manifold admits a complex generalized spin-c Killing spinor with complex Killing function $i\mu-\alpha$ if and only if it is a $\mu$-exact trans-Sasakian manifold of type $(\alpha,\mu)$. In addition, we construct a geodesically incomplete Riemannian three-manifold equipped with a complex generalized spin-c Killing spinor for which neither $\mu$ nor $\alpha$ vanishes identically, and then prove that this is not possible if the characteristic vector field of the associated trans-Sasakian structure is complete. These results settle a conjecture by Gro\ss e and Nakad (Results in Mathematics, 2015) stating that on two- and three-dimensional Riemannian manifolds the complex Killing function of any complex generalized spin-c Killing spinor must be either globally real or globally imaginary.\bigskip

\medskip


\end{abstract}

\maketitle

\setcounter{tocdepth}{1} 

\tableofcontents


\section{Introduction}


This article is devoted to the study of complex generalized spin-c Killing spinors on two- and three-dimensional Riemannian manifolds. Complex generalized spin-c Killing spinors on Riemannian manifolds were introduced by Herzlich and Moroianu~\cite{Herzlich_Moroianu_1999} for real Killing functions, and further studied for general complex Killing functions in~\cite{Grosse_Nakad_2015}, with additional results for imaginary Killing functions in~\cite{Lockman_2025,Carmona_2026}. Gro\ss e and Nakad proved, in particular, that in dimensions greater than three, a complex generalized spin-c Killing spinor necessarily reduces to either a real generalized spin-c Killing spinor or an imaginary generalized spin-c Killing spinor. In dimensions two and three, they conjectured in~\cite[Rmk.~3.5]{Grosse_Nakad_2015} that the same conclusion should hold. The conjecture is supported by~\cite[Lemma~3.4]{Grosse_Nakad_2015} which states, in dimensions greater than one, that the complex Killing function $i\mu-\alpha$ must be either real or imaginary at each point, that is, $\alpha\mu\equiv0$. As we will explain below, this claim is not correct in two dimensions.\medskip

Our approach relies on the theory of complex spinorial forms developed in~\cite{Algebraic_Complex_Square_2025, Complex_Differential_Spinors_2026}. We translate the complex generalized spin-c Killing equation into an equivalent differential equation for the complex-bilinear square of the spinor. This allows us to study the problem intrinsically on the underlying manifold, without imposing global assumptions such as simple connectivity or completeness of the Riemannian metric. This approach in dimension three is analogous to the one adopted by the authors in~\cite{Sasakian_spin-c_2026} for $\alpha$-Sasakian manifolds and in~\cite{Kenmotsu_spin-c_2026} for $\mu$-Kenmotsu manifolds.\medskip

In dimension two, we first characterize complex generalized spin-c Killing spinors in terms of their complex-bilinear squares. More precisely, Proposition~\ref{prop:complex_Killing_2d} translates the spinorial equation into a system for a section $c\in \Gamma(\cL)$ of the characteristic line bundle $\cL$ of the spin-c structure and an $\cL$-valued complex one-form $\theta$.   In particular, at a point, $c$ is zero if and only if the spinor is chiral. From this data we extract a real triple $(N,s,v)$ satisfying a first-order differential system, which allows us to describe the geometry of the solution directly on the underlying surface. On the set where the spinor is not chiral, this leads to the local normal form of Theorem~\ref{thm:surfacelocalnormalform}. In suitable local coordinates $(x,y)$, the metric takes the form:
\begin{equation*}
g=(N(x)^2-s(y)^2)(\dd x\otimes\dd x+\dd y\otimes\dd y),
\end{equation*}

and the complex-bilinear square and the unitary connection are explicitly determined by the functions $N(x)$ and $s(y)$.\medskip

We then use the real triple $(N,s,v)$ to obtain a global classification under the assumption that the Riemannian metric is complete.  Every complete solution belongs to one of two types: if the spinor is nowhere chiral, Theorem~\ref{thm:3.13} shows the surface is globally a quotient of a domain in $\mathbb{R}^2$ by a discrete group of translations and the solution is described by the two one-variable functions $N$ and $s$; if the spinor is somewhere chiral, the surface is a rotationally symmetric plane or sphere and the solution belongs to one of the polar families of Theorem~\ref{thm:main-zeros}, in which case either $\alpha\equiv0$ or $\mu\equiv0$. In particular, the only possible topological types of complete surfaces arising in this way are $\mathbb{R}^2$, $S^1\times\mathbb{R}$, $T^2$, and $S^2$.\medskip

The classification also clarifies the relation between completeness and the real and imaginary parts of the complex Killing function. Corollary~\ref{cor:surface_complete_separation} shows that, on a complete surface, the condition $\alpha\mu\equiv0$ forces either $\alpha\equiv0$ or $\mu\equiv0$, although completeness itself does not imply $\alpha\mu\equiv0$. In the compact case, Corollary~\ref{cor:surfacecompactclassification} shows that the only possibilities are rotationally symmetric spheres and tori, and that genuine mixing of $\alpha$ and $\mu$ can occur only in the toroidal case. Moreover, Corollary~\ref{cor:surfacetorusconsequences} shows that every toroidal solution has at least one point at which both coefficients vanish.\medskip

We complement these structural results with several explicit examples. Example~\ref{ex:spin-counterexample} gives a complex generalized spin Killing spinor on a Euclidean domain with $\alpha\mu\not\equiv0$, showing that the pointwise restriction asserted in~\cite[Lemma~3.4]{Grosse_Nakad_2015} already fails in the spin setting, without the presence of a non-trivial auxiliary connection; Example~\ref{ex:nowhere-vanishing} gives a spin-c example for which $\alpha\mu>0$ everywhere, whereas Example~\ref{ex:alpha.mu=0-2} has $\alpha\mu\equiv0$ while neither coefficient vanishes identically; finally, Example~\ref{example:complete_2d} provides a complete example on $\mathbb{R}^2$ with $\alpha\mu>0$ everywhere, and Example~\ref{ex:compact} gives a periodic version on the torus. Thus, even completeness or compactness forces neither the complex Killing function to be globally real or globally imaginary, nor $\alpha \mu = 0$.

\begin{table}[htbp]
\centering
\small
\renewcommand{\arraystretch}{1.2}
\setlength{\tabcolsep}{4pt}
\begin{tabular}{|p{0.22\linewidth}|p{0.40\linewidth}|p{0.31\linewidth}|}
\hline
\textbf{Result} & \textbf{Description} & \textbf{Significance} \\
\hline

Theorem~\ref{thm:surfacelocalnormalform}
&
Local classification around non-chiral points, with the Killing function and connection explicitly determined.
&
Reduces the local problem to two functions of one variable.
\\
\hline

Theorems~\ref{thm:3.13} and~\ref{thm:main-zeros}
&
Classification on complete Riemannian surfaces, which are translation quotients or rotational planes and spheres with chiral points.
&
Classifies all complete non-parallel solutions; chiral points strongly restrict the underlying geometry.
\\
\hline

Corollary~\ref{cor:surface_complete_separation}
&
On a complete surface,
$\alpha\mu\equiv0$ implies
$\alpha\equiv0$ or $\mu\equiv0$.
&
Excludes mixing with disjoint supports under completeness.
\\
\hline

Corollaries~\ref{cor:surfacecompactclassification} and~\ref{cor:surfacetorusconsequences}
&
Compact solutions occur on rotational spheres or separated tori.
Every toroidal solution has a point where $\alpha=\mu=0$.
&
Confines mixed compact solutions to tori and constrains
the zeros of their Killing function.
\\
\hline

Examples~\ref{ex:spin-counterexample}--\ref{ex:compact}
&
Constructs local, complete, and compact mixed solutions,
including a spin example and a complete example with
$\alpha\mu>0$.
&
Disproves the two-dimensional pointwise vanishing assertion
and the real-or-imaginary conjecture, even under compactness.
\\
\hline
\end{tabular}

\vspace{5pt}

\caption{Main results in two dimensions. Throughout, the Killing function $i\mu-\alpha$ is not identically zero,
and existence allows a suitable spin-c structure and connection.}
\label{tab:surface-results}
\end{table}

We then proceed to consider the three-dimensional case. In dimension three, the picture is different. We first show that the complex generalized spin-c Killing equation admits a precise geometric interpretation in terms of trans-Sasakian geometry. More precisely, our first main result in this dimension establishes a correspondence between complex generalized spin-c Killing spinors and a distinguished class of trans-Sasakian structures.

\begin{thm}
\label{thm:ifandonlyif}
A three-dimensional Riemannian manifold $(M,g)$ admits a complex generalized spin-c Killing spinor with complex Killing function $i\mu-\alpha$ if and only if it is a trans-Sasakian manifold of type $(\alpha,\mu)$ such that the one-form $\mu\zeta$ is exact.
\end{thm}

Trans-Sasakian manifolds form a class of almost contact metric manifolds introduced by Oubiña in~\cite{Oubina1985}, corresponding to the class $\mathcal{C}_5\oplus\mathcal{C}_6$ in the Chinea--González classification~\cite{ChineaGonzalez1990}. Marrero proved in~\cite{Marrero1992} that, in dimensions greater than three, a trans-Sasakian manifold of type $(\alpha,\mu)$ is necessarily of type $(0,0)$, $(\alpha,0)$, or $(0,\mu)$, corresponding respectively to the cosymplectic (or co-Kähler), $\alpha$-Sasakian, and $\mu$-Kenmotsu cases. In dimension three this rigidity no longer holds for arbitrary trans-Sasakian structures. Indeed, in~\cite{Marrero1992} the author constructs examples which are neither $\alpha$-Sasakian nor $\mu$-Kenmotsu. However, these examples do not satisfy the exactness condition for $\mu\zeta$ appearing in Theorem~\ref{thm:ifandonlyif}. On the other hand, by Theorem~\ref{thm:ifandonlyif}, the counterexample constructed in Section~\ref{sec:counterexample} yields an incomplete $\mu$-exact trans-Sasakian three-manifold for which neither $\alpha$ nor $\mu$ vanishes identically. The situation changes once a natural completeness assumption is imposed, leading to our second main result in this dimension.

\begin{thm}
\label{thm:completeReebnogo}
Let $(M,g,\varphi,\xi,\zeta)$ be a connected three-dimensional trans-Sasakian manifold of type $(\alpha,\mu)$. Assume that the one-form $\mu\zeta$ is closed and that the characteristic vector field $\xi$ is complete. Then either $\alpha\equiv0$ or $\mu\equiv0$ everywhere on $M$.
\end{thm}

Since the characteristic vector field $\xi$ has unit norm and satisfies $\nabla^g_{\xi}\xi=0$, completeness of the Riemannian metric implies completeness of $\xi$. We therefore obtain the following consequence.

\begin{cor}
Let $(M,g,\varphi,\xi,\zeta)$ be a connected three-dimensional trans-Sasakian manifold of type $(\alpha,\mu)$. Assume that the one-form $\mu\zeta$ is closed and that $g$ is complete. Then either $\alpha\equiv0$ or $\mu\equiv0$ everywhere on $M$.
\end{cor}

Combining this corollary with Theorem~\ref{thm:ifandonlyif}, we conclude that a complete three-dimensional spin-c Riemannian manifold carrying a complex generalized spin-c Killing spinor has complex Killing function which is either globally real or globally imaginary. In particular, there are no complete, and hence no compact, three-dimensional counterexamples to the conjecture of Gro\ss e and Nakad.\medskip

The paper is organized as follows. In Section~\ref{sec:pre}, we introduce complex generalized spin-c Killing spinors and fix the conventions used throughout the paper. In Section~\ref{sec:2d}, we study the two-dimensional case. We derive the differential system satisfied by the complex-bilinear square, obtain local normal forms, classify the complete solutions, and construct the explicit examples described above. In Section~\ref{sec:equivalence_trans-Sasakian}, we prove Theorem~\ref{thm:ifandonlyif}, establishing the correspondence between complex generalized spin-c Killing spinors and $\mu$-exact trans-Sasakian structures. In Section~\ref{sec:counterexample}, we construct an incomplete three-dimensional counterexample to the Gro\ss e--Nakad conjecture, while in Section~\ref{sec:completeness} we prove Theorem~\ref{thm:completeReebnogo} and derive the corresponding completeness result. Finally, Appendix~\ref{Ap:Lemma3.4} analyzes the proof of~\cite[Lemma~3.4]{Grosse_Nakad_2015}, verifies the two-dimensional counterexample directly using the classical spinor formalism, and proves the corrected pointwise statement in dimensions greater than two.

\begin{ack}
The work of JLCJ has been supported by a Riemann fellowship at the Riemann Center for Geometry and Physics, Leibniz Universität Hannover (Germany) and by the project PID2024-156578NB-I00 (Spain). The work of AGG is supported by the Scuola Internazionale Superiore di Studi Avanzati (SISSA). The work of CSS was partially supported by the research grant PID2023-152822NB-I00 of the Ministry of Science of the government of Spain.
\end{ack}


\section{Preliminaries}
\label{sec:pre}

 
Let $(M,g)$ be a connected oriented Riemannian manifold of dimension $d\geq2$ with volume form $\mathrm{vol}_g$. Let $S$ be a bundle of irreducible complex spinors over the bundle of Clifford algebras of $(M,g)$. Throughout Sections~2--4, we use the positive Clifford product convention, denoted by \emph{dot}~$\cdot$:
\begin{equation*}
X\cdot Y + Y \cdot X = 2g(X,Y).
\end{equation*}

The existence of such a bundle $S$ is equivalent to the existence of a $\mathrm{Spin}^c(d)$-structure $Q$ on $(M,g)$ to which $S$ is associated via the standard associated vector bundle construction:
\begin{equation*}
S=Q\times_{\mathfrak{r}}\C^{2^{\lfloor d/2\rfloor}},
\end{equation*}

where $\mathfrak{r}$ denotes the restriction to the spin-c group $\mathrm{Spin}^c(d)\subset\C\mathrm{l}(d)$ of an irreducible representation $\gamma$ of $\C\mathrm{l}(d)$. In odd dimension, there are two inequivalent irreducible $\C\mathrm{l}(d)$-representations. For convenience, we choose $\gamma$ to satisfy $\gamma(i^{(d-1)/2} \mathrm{vol}_g) = -\Id$.\medskip

In particular, $S$ admits a non-degenerate complex-bilinear pairing~\cite{AC97,ACDVP05,Algebraic_Complex_Square_2025}:
\begin{equation*}
\scB\colon S\otimes S\to\mathcal{L}
\end{equation*}

such that, for some $\delta_1$, $\delta_2 \in \{ \pm1\}$:
\begin{equation*}
\scB(\eta_1,\eta_2)=\delta_1\scB(\eta_2,\eta_1),\qquad\scB(\theta\cdot\eta_1,\eta_2)=\delta_2\scB(\eta_1,\theta\cdot\eta_2)
\end{equation*}

for all $\theta\in\Omega^1_\C(M)$ and $\eta_1$, $\eta_2\in\Gamma(S)$. In dimension two, by~\cite[Table~3, p.~19]{Algebraic_Complex_Square_2025}, we choose $\scB$ to be symmetric of positive adjoint type, that is, $\delta_1=\delta_2=1$. In dimension three, by~\cite[Table~7, p.~29]{Algebraic_Complex_Square_2025}, we have that $\scB$ is skew-symmetric of negative adjoint type, that is, $\delta_1=\delta_2=-1$.\medskip

Let $\mathcal{L}\to M$ be the characteristic Hermitian complex line bundle on $M$ associated with the given $\mathrm{Spin}^c(d)$-structure $Q$. We will denote the corresponding principal $\mathrm{U}(1)$-bundle by $P\to M$. Since $S$ is associated with a spin-c structure $Q$, the combination of the lift of the Levi-Civita connection $\nabla^g$ of $(M,g)$ to $S$ together with the choice of a connection $A\in\Omega^1(P,i\R)$ on the characteristic $\mathrm{U}(1)$-bundle $P$ defines a covariant derivative $\nabla^{g,A}$ on $S$.

\begin{remark}
It is well-known that every orientable manifold $M$ of dimension two or three admits a spin structure. However, a spin-c structure on $M$ reduces to a spin structure if and only if the characteristic complex line bundle $\mathcal{L}$ is trivial. Even when $\mathcal{L}$ is trivial, the connection one-form $A$ on $\mathcal{L}$ may have non-trivial holonomy. In that case, the connection $\nabla^{g,A}$ does not reduce to the lift of the Levi-Civita connection $\nabla^g$.
\end{remark}

\begin{definition}\label{def:complex-Kill-eq}
A \emph{complex generalized spin-c Killing spinor} on $(M,g)$ with \emph{complex Killing function} $i\mu-\alpha$, where $\alpha$, $\mu\in C^\infty(M,\R)$, is a  non-trivial section $\eta\in\Gamma(S)$ of a bundle of irreducible complex spinors $S$ on $(M,g)$ satisfying the following equation:
\begin{equation*}
\nabla^{g,A}_X\eta=\frac{i\alpha+\mu}{2}X\cdot\eta,\qquad X\in\Gamma(TM)
\end{equation*}

for a connection $A\in\Omega^1(P,i\R)$ on the characteristic $\U(1)$-bundle $P$ of the spin-c structure $Q$.
\end{definition}

Since $(M,g)$ is connected, by~\cite[Lemma~3.1(iv)]{Grosse_Nakad_2015}, we have that $\eta$ is nowhere vanishing.

\begin{remark}\label{rmk:convenciones}
A negative Clifford product convention, denoted in this remark by \emph{bullet} $\bullet$, is defined by setting $X\bullet = -iX\cdot$. Consequently, under this identification, the spinor bundle and the spinorial covariant derivative are unchanged and the Killing equation in Definition~\ref{def:complex-Kill-eq} is given by:
\begin{equation*}
\nabla^{g,A}_X\eta=\frac{i\alpha+\mu}{2}X\cdot\eta=\frac{i\mu-\alpha}{2}X\bullet\eta,\qquad X\in\Gamma(TM).
\end{equation*}

In Appendix~\ref{Ap:Lemma3.4}, we will follow the conventions and notation of~\cite{Grosse_Nakad_2015}. In particular, we will use the negative Clifford multiplication, which will be denoted by $\cdot$ instead of $\bullet$. Furthermore,  the complex Killing function is denoted by $K$, and it is related to our definitions by $K = \frac{i\mu-\alpha}{2}$.
\end{remark}


\section{Generalized Killing spinors on Riemannian surfaces}
\label{sec:2d}


In this section, we consider complex generalized spin-c Killing spinors on Riemannian surfaces and obtain an exhaustive classification when the underlying Riemannian surface is complete. Throughout this section, $(M,g)$ is a connected Riemannian surface with no boundary, and we will study the geometry of $(M,g)$ when the Killing function $i\mu-\alpha \in C^{\infty}(M,\C)$ is not identically zero. The case $i\mu-\alpha\equiv0$ reduces to that of a parallel spin-c spinor, which was solved in \cite{Moroianu1997}, from which it follows that every Riemannian surface admits a spin-c parallel spinor. 


\subsection{The square of a generalized spin-c Killing spinor}


We begin by computing the complex-bilinear square of a generalized spin-c Killing spinor on a Riemannian surface, our main working tool throughout this section. The Hermitian metric $(-,-)_{\cL}$ is linear in its first entry and conjugate-linear in its second. Let $\cL\to M$ be a Hermitian complex line bundle on $(M,g)$. For $\cL$-valued complex one-forms, the induced Hermitian metric and the complex-bilinear pairing are, respectively, defined as follows:
\begin{equation*}
(\varphi,\psi)_{\cL,g}:=\sum_{j=1}^2\bigl(\varphi(e_j),\psi(e_j)\bigr)_{\cL} \in C^{\infty}(M,\C), \qquad \langle\varphi,\psi\rangle:=\sum_{j=1}^2\varphi(e_j)\otimes\psi(e_j)\in\Gamma(\cL \otimes \cL),    
\end{equation*}

where $(e_1,e_2)$ is any real $g$-orthonormal frame. The first contraction is scalar-valued and defines $|\varphi|^2=(\varphi,\varphi)_{\cL,g}$; the second contains no complex conjugation. Given a unitary connection $A$ on $\cL$, the connection $\nabla^{g,A}$ on $\wedge T^*_\C M\otimes\cL$ is defined as follows on homogeneous elements $\xi\otimes \mathfrak{l} \in \Omega_\C(M,\cL)$:
\begin{equation*}
\nabla^{g,A}_X(\xi\otimes\mathfrak{l}) = (\nabla^g_X\xi)\otimes\mathfrak{l} + \xi\otimes(\nabla^A_X\mathfrak{l}).
\end{equation*}

This connection preserves the induced Hermitian metric and is compatible with the complex-bilinear contraction and the induced connection on $\cL\otimes\cL$. Complex conjugation sends $\cL$-valued forms to $\overline{\cL}$-valued forms. When conjugate line-bundle factors occur, we contract them through the parallel map:
\begin{equation*}
\cL\otimes\overline{\cL}\longrightarrow M\times\mathbb C, \qquad \mathfrak l_1\otimes\overline{\mathfrak{l}}_2\mapsto(\mathfrak l_1,\mathfrak l_2)_{\cL},
\end{equation*}

using the interchange of the line factors if they occur in the opposite order. In particular, $\varphi\wedge\overline{\varphi}$ defines a complex two-form after the aforementioned contraction. The ordinary complex one-form $(\varphi,\mathfrak l)_{\cL} \in\Omega_\C^1(M)$ is defined by:
\begin{equation*}
(\varphi,\mathfrak l)_{\cL}(X):=(\varphi(X),\mathfrak l)_{\cL},\qquad X\in\Gamma(T_\C M).
\end{equation*}

This complex one-form on $M$ will be crucial for the upcoming classification.

\begin{prop}
\label{prop:algebraicsquare}
Let $(M,g)$ be an oriented Riemannian surface and let $\mathcal L$ be a Hermitian complex line bundle on $M$. An $\mathcal L$-valued complex exterior form $\rho \in \Omega_\C(M,\cL)$ is the complex-bilinear square of a nowhere vanishing irreducible complex spinor for some bundle of irreducible complex spinors on $(M,g)$ if and only if:
\begin{equation}
\label{eq:surfacealgebraic}
\rho = c+\theta,\qquad \langle\theta,\theta\rangle=c\otimes c,
\end{equation}
where $c\in\Gamma(\mathcal L)$ and $\theta\in\Omega^1_{\mathbb C}(M,\mathcal L)$ is nowhere vanishing.
\end{prop}

\begin{proof}
Choose a symmetric admissible complex-bilinear pairing of positive adjoint type. Pointwise, $\rho \in \Omega_\C(M,\cL)$ is the square of a nowhere vanishing irreducible complex spinor if and only if the algebraic conditions in \cite[Thm.~2.7]{Complex_Differential_Spinors_2026} hold. Adapted to the present situation, they reduce to:
\begin{equation*}
\rho \diamond \rho = 2\rho^{(0)}\rho,\qquad \rho\diamond\beta\diamond\rho = 2(\rho \diamond \beta)^{(0)}\rho, \qquad \tau(\rho) = \rho   
\end{equation*}

for a complex exterior form $\beta$ such that $(\rho\diamond\beta)^{(0)} \neq 0$. The linear equation eliminates the degree-two component, so we can write $\rho = c + \theta$ for a section $c\in \Gamma(\cL)$ and an $\cL$-valued one-form $\theta\in\Omega^1_{\mathbb C}(M,\mathcal L)$. The first equation then becomes $\langle\theta,\theta\rangle=c\otimes c$. Since the spinor is nowhere vanishing by assumption, we have $\theta\neq 0$. To apply \cite[Thm.~2.7]{Complex_Differential_Spinors_2026}, we take $\beta=\overline{\theta}$. Then:
\begin{equation*}
(\rho\diamond\overline{\theta})^{(0)}=|\theta|^2>0.
\end{equation*}

With this choice of $\beta$, the second equation above follows automatically from the first. Globalizing this result by properly taking into account that $c$ is a section of $\cL$ and $\theta$ is an $\cL$-valued one-form, we obtain the statement in the proposition.
\end{proof}

\begin{remark}
The section $c\in\Gamma(\cL)$ may have zeros. At a point where $c=0$, the non-zero $\cL$-valued one-form $\theta$ is isotropic and belongs to one of the two complex Hodge eigenspaces, tensored with $\cL$. Such points are precisely the pointwise chiral locus of the corresponding spinor, see \cite[Cor.~5.2]{Algebraic_Complex_Square_2025}.
\end{remark}

\begin{prop}
\label{prop:complex_Killing_2d}
An oriented Riemannian surface $(M,g)$ admits a generalized Killing spinor with complex Killing function $i\mu-\alpha$ if and only if there exist a Hermitian line bundle $\mathcal L$, a unitary connection $A$ on $\cL$, a section $c\in\Gamma(\cL)$, and a nowhere vanishing $\cL$-valued complex one-form $\theta$ satisfying:
\begin{equation}
\label{eq:surfacecovariant}
\nabla^A_Xc=(\mu+i\alpha)\theta(X),\qquad
\nabla^{g,A}_X\theta=(\mu+i\alpha)X^\flat\otimes c,\qquad
\langle\theta,\theta\rangle=c\otimes c
\end{equation}
for every real vector field $X\in\Gamma(TM)$.
\end{prop}

\begin{proof}
By \cite[Lemma~5.3 and Cor.~5.7]{Complex_Differential_Spinors_2026}, a non-zero irreducible complex spinor $\eta$ on $(M,g)$ is Killing with complex Killing function $i\mu -\alpha$ if and only if its square $\rho = c + \theta$ satisfies:
\begin{equation*}
\nabla^{g,A}_X \rho = \frac{\mu+i\alpha}{2} (X^\flat\diamond\rho + \rho \diamond X^\flat).
\end{equation*}

By Proposition~\ref{prop:algebraicsquare}, we can write $\rho = c + \theta$ and, separating by degree, we obtain all equations in~\eqref{eq:surfacecovariant}.  
\end{proof}

The Hodge operator acts complex-linearly on the form factor and as the identity on $\cL$. Note that for a positively oriented real orthonormal coframe $\{e^1 , e^2\}$, we have $\ast e^1=e^2$, $\ast e^2=-e^1$, and $\mathrm{vol}_g =e^1\wedge e^2$.

\begin{remark} 
Let $\dd_A$ denote the covariant exterior derivative and let $\dd_A^*$ denote its formal adjoint for the induced $L^2$ Hermitian pairing, tested on compactly supported forms. Since $M$ has real dimension two and $A$ is unitary, $\dd_A^{\ast} = - \ast \dd_A \ast$. Alternating and taking the negative trace of the second covariant equation yields:
\begin{equation*} 
\dd_Ac=(\mu+i\alpha)\theta,\qquad
\dd_A\theta=0,\qquad
\dd_A^*\theta=-2(\mu+i\alpha)c
\end{equation*}

as a consequence of~\eqref{eq:surfacecovariant}.
\end{remark}


\subsection{Real invariants associated with a generalized Killing spinor}


A \emph{real triple} on $(M,g)$ is a triple $(N,s,v)$, where $N$, $s\in C^\infty(M,\R)$ and $v\in\Omega^1(M)$, satisfying $N>0$ and $16|v|^2=N^2-s^2$ and its \emph{regular locus} is $M^\circ:=\{v\neq0\}=\{N>|s|\}$. For instance, given a section $c \in \Gamma(\cL)$ and an $\cL$-valued one-form $\theta\in\Omega^1_{\mathbb C}(M,\mathcal L)$ satisfying~\eqref{eq:surfacealgebraic}, with $\theta$ nowhere vanishing, we define the following two real functions and a real one-form on $M$:
\begin{gather*}
N:=\sqrt{2(|c|^2+|\theta|^2)} \in C^{\infty}(M,\R),\qquad s:=\frac{2i*(\theta\wedge\overline{\theta})}{N} \in C^{\infty}(M,\R),\\
v:=\frac{\mathrm{Re}(\theta,c)_{\cL}}{N}\in\Omega^1(M).
\end{gather*}

Since $\theta$ is nowhere vanishing, $N$ is positive, and all three expressions are globally defined on $M$. We will refer to such a tuple $(N,s,v)$ as the \emph{real triple} associated with a given nowhere vanishing irreducible complex spinor on $(M,g)$, and the following lemma proves that it is indeed a real triple.

\begin{lemma}
\label{lemma:surfacerealalgebra}
Let $(N,s,v)$ be the real triple associated with a nowhere vanishing irreducible complex spinor on $(M,g)$. Then, for every $X,Y\in\Gamma(TM)$, we have:
\begin{gather}
16|v|^2 =N^2-s^2,\qquad 4|v|^2 = |c|^2 ,\qquad (\theta,c)_{\cL} = N v + i s\ast v , \qquad 4|\theta|^2 = N^2+s^2,
\label{eq:surfacerealalgebra}\\
\mathrm{Re}(\theta(Y),\theta(X))_{\cL} = \frac{s^2}{4}g(X,Y)+4v(X)v(Y),\qquad
\mathrm{Im}(\theta(Y),\theta(X))_{\cL} =\frac{Ns}{4} \mathrm{vol}_g(X,Y).
\label{eq:surfacehermitiantensor}
\end{gather}
\end{lemma}

\begin{proof}
Choose a real, positively oriented orthonormal coframe $\{e^1, e^2\}$ and a local unit section $\mathfrak l$ of $\cL$. We write:
\begin{equation*}
c=k\mathfrak l,\qquad \theta=(\theta_1e^1+\theta_2e^2)\otimes\mathfrak l
\end{equation*}

for local functions $k$, $\theta_1$, and $\theta_2$. The algebraic relation $\langle \theta , \theta \rangle = c\otimes c$ reduces to $\theta_1^2+\theta_2^2=k^2$. A direct computation gives:
\begin{align*}
(i*(\theta\wedge\overline{\theta}))^2  &=(|\theta_1|^2+|\theta_2|^2)^2 - |\theta_1^2+\theta_2^2|^2 = |\theta|^4-|c|^4,\\
\vert \mathrm{Re}(\theta , c)_{\cL}\vert^2 &=\sum_{j=1}^2\mathrm{Re}(\overline{k} \theta_j)^2 =\frac{|c|^2(|c|^2+|\theta|^2)}2.
\end{align*}

On the other hand:
\begin{equation*}
\langle(\theta,c)_{\cL},(\theta,c)_{\cL} \rangle = |c|^4
\end{equation*}

as well as:
\begin{equation*}
\vert \mathrm{Re}(\theta,c)_{\cL}\vert^2 = \frac{N^2|c|^2}{4}, \qquad  \ast (\mathrm{Re}(\theta,c)_{\cL} \wedge\mathrm{Im}(\theta,c)_{\cL}) =\frac{|c|^2}{2}\,i*(\theta\wedge\overline{\theta}) =\frac{Ns|c|^2}{4}.
\end{equation*}

The preceding two identities give:
\begin{equation*}
\mathrm{Im}(\theta,c)_{\cL}
 =\frac{s}{N}*\mathrm{Re}(\theta,c)_{\cL}=s*v.
\end{equation*}
Combining these identities together with $\mathrm{Re}(\theta,c)_{\cL}=Nv$ gives all equations in~\eqref{eq:surfacerealalgebra}.\medskip

Suppose $c\neq0$ locally. We have:
\begin{equation*}
\theta=\frac{(\theta,c)_{\cL}}{|c|^2}\otimes c =\frac{Nv+i s*v}{|c|^2}\otimes c.    
\end{equation*}

Consequently, for every $X,Y\in\Gamma(TM)$ we obtain:
\begin{align*}
4 |c|^2 \mathrm{Re}(\theta(Y),\theta(X))_{\cL} &= 4 (N^2v(X)v(Y)+s^2(*v)(X)(*v)(Y))\\
&= s^2 |c|^2 g(X,Y) + 16 |c|^2 v(X)v(Y),\\
4 |c|^2 \mathrm{Im}(\theta(Y),\theta(X))_{\cL} &= 4Ns (v\wedge*v)(X,Y) = N|c|^2 s \,\mathrm{vol}_g(X,Y),
\end{align*}

where we are using that $v\otimes v+(*v)\otimes(*v)=|v|^2g$. Using $4|c|^2=N^2-s^2$ gives~\eqref{eq:surfacehermitiantensor} on the locus where $c\neq 0$. To recover the same identities at $c=0$, condition $\langle \theta , \theta \rangle = c\otimes c$ and the assumption $\theta\neq 0$ give $\theta_2=\pm i\theta_1$ with $\theta_1\neq0$. Hence:
\begin{equation*}
N=2|\theta_1|,\qquad s=\pm N,\qquad v=0,    
\end{equation*}

with matching signs, and direct evaluation gives:
\begin{equation*}
\mathrm{Re}(\theta(Y),\theta(X))_{\cL}=|\theta_1|^2g(X,Y),\qquad \mathrm{Im}(\theta(Y),\theta(X))_{\cL}  =\pm|\theta_1|^2\mathrm{vol}_g (X,Y).
\end{equation*}

These are precisely~\eqref{eq:surfacehermitiantensor}. 
\end{proof}


\subsection{The differential system of a real triple} 


We now derive the differential system satisfied by the real triple $(N,s,v)$ associated with a generalized Killing spinor on $(M,g)$ with square $\rho = c + \theta$.

\begin{prop}
The real triple $(N, s, v)$ associated with a generalized Killing spinor on $(M,g)$ with Killing function $i\mu-\alpha$ satisfies the differential system:
\begin{gather}
\dd N=4\mu v,\qquad \dd s=4\alpha*v,\qquad \dd |c|^2=2N\mu\, v - 2s\alpha*v,
\label{eq:surfacerealsystem}\\
4 \nabla^g v=\mu N g + \alpha s\, \mathrm{vol}_g,\qquad 4 \nabla^g *v=\mu N \,\mathrm{vol}_g - \alpha s\,g,
\label{eq:nabla-v}\\
2 \mathscr{L}_{v^\sharp}g=\mu N g,\qquad 2 \mathscr{L}_{(*v)^\sharp}g=- \alpha s g, \qquad 2 \dd v=\alpha s \mathrm{vol}_g, \qquad 2\dd (*v)=\mu N \mathrm{vol}_g,
\label{eq:Lie-dd-v}\\
[v^\sharp,(*v)^\sharp]=0,\qquad \dd N((*v)^\sharp)=0,\qquad \dd s(v^\sharp)=0.
\label{eq:surfaceconformalfields}
\end{gather}
\end{prop}

\begin{proof}
The metric compatibility of $\nabla^{g,A}$ and the tensor-product definition of $(-,-)_{\cL,g}$ give, for every real vector field $X$, the following equations:
\begin{align*}
X(|c|^2) &= 2\mathrm{Re} ((\mu+i\alpha)(\theta(X),c)_{\cL}),\\
X(|\theta|^2) &=2\mathrm{Re} ((\mu+i\alpha)(X^\flat\otimes c,\theta)_{\cL,g}) = 2\mathrm{Re}((\mu-i\alpha)(\theta(X),c)_{\cL}).
\end{align*}

In the last equality, we use $(X^\flat\otimes c,\theta)_{\cL,g}=\overline{(\theta(X),c)}_{\cL}$, obtained by evaluating the induced Hermitian metric in a real orthonormal frame. Hence:
\begin{equation*} 
X\bigl(|c|^2+|\theta|^2\bigr)  =4 \mu\,\mathrm{Re}(\theta(X) , c)_{\cL}.
\end{equation*}

Since $N^2=2(|\theta|^2+|c|^2)$ in~\eqref{eq:surfacerealalgebra}, it follows that $N\dd N=4\mu\,\mathrm{Re}(\theta,c)_{\cL}$ and hence $\dd N=4\mu v$.\medskip 

The tensor-product rule, using the conjugate connection on $\overline{\cL}$ and the parallel Hermitian contraction, gives:
\begin{align*}
X(Ns) &=2i\ast ((\mu+i\alpha)X^\flat\wedge\overline{(\theta,c)}_{\cL}-(\mu-i\alpha)X^\flat\wedge(\theta,c)_{\cL})\\
&=4\mu s\,v(X)+4\alpha N\,(*v)(X).
\end{align*}

Here we have used~\eqref{eq:surfacerealalgebra} and the identities
$*(X^\flat\wedge v)=-(*v)(X)$ and $*(X^\flat\wedge*v)=v(X)$. Subtracting $s\dd N$ and dividing by the positive function $N$ proves $\dd s=4\alpha*v$. Furthermore, since $4|c|^2 = N^2 - s^2$ in~\eqref{eq:surfacerealalgebra}, we have the last identity in~\eqref{eq:surfacerealsystem}.\medskip 

For the covariant derivative equations, for every $X,Y\in\Gamma(TM)$, we compute:
\begin{align*}
(\nabla^g_X(\theta,c)_{\cL})(Y)&=((\nabla^{g,A}_X\theta)(Y),c)_{\cL}+(\theta(Y),\nabla^A_Xc)_{\cL}\\
&= (\mu+i\alpha)|c|^2g(X,Y)+(\mu-i\alpha)(\theta(Y),\theta(X))_{\cL}.
\end{align*}

Taking real parts and using the pointwise identities~\eqref{eq:surfacehermitiantensor} yields:
\begin{equation*}
\nabla^g\mathrm{Re}(\theta,c)_{\cL}=\frac{\mu N^2}{4}g+4\mu v\otimes v+\frac{\alpha Ns}{4}\mathrm{vol}_g .
\end{equation*}

Differentiating the identity $Nv=\mathrm{Re}(\theta,c)_{\cL}$ that was obtained in Lemma~\ref{lemma:surfacerealalgebra}, and using $\dd N=4\mu v$ it follows that the $v\otimes v$ terms cancel, proving the first identity in~\eqref{eq:nabla-v}. Since the Hodge star operator is parallel and squares to minus the identity on one-forms, we obtain the second identity in~\eqref{eq:nabla-v}. Then we use $\mathscr{L}_X g = 2\mathrm{Sym}(\nabla^g X^\flat)$ and $\dd X^\flat = 2\mathrm{Skew}(\nabla^g X^\flat)$ to conclude~\eqref{eq:Lie-dd-v}. The last two identities in~\eqref{eq:surfaceconformalfields} follow, respectively, from evaluating $\dd N=4\mu v$ and $\dd s=4\alpha*v$ on $(\ast v)^{\sharp}$ and $v^\sharp$ and using $\langle v,*v\rangle=0$. Finally, we use that $\nabla^g$ is torsion-free and~\eqref{eq:nabla-v} to prove the first equation of~\eqref{eq:surfaceconformalfields}. 
\end{proof}

\begin{lemma}\label{lem:surface_zero_set}
Let $(N, s, v)$ be the real triple associated with a generalized Killing spinor with $i\mu-\alpha\not\equiv0$ on $(M,g)$. Then $v$ is zero at $m\in M$ if and only if $c(m)=0$, and the zero set of $c$, equivalently of $v$, is closed and discrete in $M$. 
\end{lemma}

\begin{proof}
If $v$ vanished identically, then~\eqref{eq:surfacerealalgebra} would give $s^2=N^2>0$, and the symmetric and skew-symmetric parts of the first equation in~\eqref{eq:nabla-v} would force $\mu=\alpha=0$ everywhere, a contradiction. Thus $v^\sharp$, and consequently $(\ast v)^{\sharp}$, are non-trivial conformal Killing vector fields on $(M,g)$. Choosing a local conformal coordinate $z=x+iy$, we write $v^\sharp = v^x\partial_x + v^y \partial_y$. The trace-free part of the conformal Killing equation gives:
\begin{equation*}
\partial_xv^x=\partial_yv^y,\qquad \partial_y v^x=-\partial_x v^y.
\end{equation*}

Hence $v^x+i v^y$ is holomorphic. The identity theorem and connectedness imply that a conformal field vanishing on a nonempty open subset vanishes identically. Therefore the local holomorphic coefficient of $v^\sharp$ is not identically zero on any coordinate disk, and consequently its zeros are isolated. Since $4|v|^2=|c|^2$ by Lemma~\ref{lemma:surfacerealalgebra}, these zeros are precisely those of $c$. Hence, the common zero set of $c$ and $v$ is closed and discrete.
\end{proof}

We denote by $M^{\circ}\subset M$ the complement in $M$ of the common zero set of $c$ and $v$. By the previous lemma, $M^\circ\subset M$ is a connected, open, and dense subset of $M$.


\subsection{Adapted coframes and local normal forms}  


Let $(N, s, v)$ be the real triple associated with a generalized Killing spinor on $(M,g)$. On the non-zero locus $M^\circ\subset M$ of $c$, we introduce the following positively oriented coframe:
\begin{equation}
\label{eq:surfaceadaptedcoframe}
e^1:=\frac{2v}{|c|},\qquad e^2:=\frac{2*v}{|c|}.
\end{equation}

By Lemma~\ref{lemma:surfacerealalgebra}, $4|v|^2=|c|^2$, so $\{e^1,e^2\}$ is a positively oriented orthonormal coframe. We will refer to such $\{e^1,e^2\}$ as the \emph{canonical coframe} of the real triple $(N, s, v)$. For further reference, we define the functions:
\begin{equation*}
\phi:=\operatorname{arctanh}\left(\frac{s}{N}\right),\qquad f:=\log|c|.
\end{equation*}

Note that:
\begin{equation*} 
(\theta,c)_{\cL} = |c|^2(\cosh(\phi)\,e^1+i\sinh(\phi)\,e^2),\qquad
\cosh(\phi) = \frac{N}{2|c|},\qquad \sinh(\phi)=\frac{s}{2|c|}.
\end{equation*}

Equivalently, $\theta=(\cosh(\phi)\,e^1+i\sinh(\phi)\,e^2)\otimes c$ on the regular locus $M^{\circ}\subset M$. Note that, since $c$ is nowhere vanishing on $M^{\circ}$, it follows that $\cL$ is trivializable when restricted to $M^{\circ}\subset M$.

\begin{prop}
\label{prop:complex_Killing_2d_frame}
Let $(N, s, v)$ be the real triple associated with a generalized Killing spinor on $(M,g)$ with complex-bilinear square $\rho = c + \theta$. Then its canonical coframe $\{e^1,e^2\}$ satisfies the following differential system on $M^{\circ}\subset M$:
\begin{equation}
\label{eq:surfaceframesystem}
\begin{gathered}
\dd e^1 = \alpha\sinh(\phi)\,e^1\wedge e^2,\qquad \dd e^2 = \mu\cosh(\phi)\,e^1\wedge e^2,\\
\dd\phi = -\mu\sinh(\phi)\,e^1+\alpha\cosh(\phi)\,e^2,\qquad \dd f = \mu\cosh(\phi)\,e^1-\alpha\sinh(\phi)\,e^2.
\end{gathered}
\end{equation}

Conversely, given a positively oriented orthonormal coframe $\{e^1,e^2\}$ and real functions $\phi,f,\alpha,\mu$, with $i\mu-\alpha \not\equiv 0$, satisfying the differential system~\eqref{eq:surfaceframesystem} on a connected surface $(M,g)$, there exists a generalized Killing spinor whose complex-bilinear square $\rho=c+\theta$ is given by:
\begin{equation*}
c = e^f \mathfrak l,\qquad \theta=e^f(\cosh(\phi)\,e^1+i\sinh(\phi)\,e^2)\otimes\mathfrak l
\end{equation*}

for a unit section $\mathfrak{l}$ of the trivial complex Hermitian line bundle and the unitary connection $A$ defined by:
\begin{equation*}
\nabla^A\mathfrak l =i(\alpha\cosh(\phi)\,e^1+\mu\sinh(\phi)\,e^2)\otimes\mathfrak l.
\end{equation*}
\end{prop}

\begin{proof}
Since $4|c|^2=N^2-s^2$, the differential system~\eqref{eq:surfacerealsystem} implies:
\begin{equation*}
\dd(|c|^2)=2\mu N v-2\alpha s*v,
\qquad
\dd(\log|c|)=\frac{\mu N v-\alpha s*v}{|c|^2}.
\end{equation*}

In combination with~\eqref{eq:surfacerealsystem} and~\eqref{eq:surfaceadaptedcoframe}, these identities give the two structure equations and the equation for $\dd f$. Differentiating $\tanh(\phi)=s/N$ gives the equation for $\dd\phi$.\medskip

For the converse, given $\{e^1 , e^2\}$ and real functions $\phi,f,\alpha,\mu$ as specified in the statement, we define:
\begin{equation*}
\varpi:=\alpha\sinh(\phi)\,e^1+\mu\cosh(\phi)\,e^2,
\qquad \beta:=\alpha\cosh(\phi)\,e^1+\mu\sinh(\phi)\,e^2.
\end{equation*}

The first two equations of~\eqref{eq:surfaceframesystem}, together with Cartan's structure equations, give:
\begin{equation*}
\nabla^ge^1=\varpi\otimes e^2,\qquad
\nabla^ge^2=-\varpi\otimes e^1.
\end{equation*}

Now define the unitary connection $A$ by $\nabla^A\mathfrak{l}=i\beta\otimes\mathfrak{l}$. On the one hand, we have:
\begin{equation*}
\nabla^A c =e^f(\dd f+i\beta)\otimes\mathfrak l = (\mu+i\alpha)e^f
(\cosh(\phi)\,e^1+i\sinh(\phi)\,e^2)\otimes\mathfrak l = (\mu+i\alpha)\theta.
\end{equation*}

On the other hand, the equations for $\dd f,\dd\phi$ and the given formulas for $\varpi,\beta$ give the following identities:
\begin{align*}
\cosh(\phi)(\dd f+i\beta)+\sinh(\phi)\,\dd\phi-i\sinh(\phi)\,\varpi&=(\mu+i\alpha)e^1,\\
i\sinh(\phi)(\dd f+i\beta)+i\cosh(\phi)\,\dd\phi+\cosh(\phi)\,\varpi&=(\mu+i\alpha)e^2.
\end{align*}

These are the two coefficient one-forms in the tensor-product covariant derivative of:
\begin{equation*}
\theta := e^f(\cosh(\phi)\,e^1+i\sinh(\phi)\,e^2)\otimes\mathfrak{l}.
\end{equation*}

Therefore:
\begin{equation*}
\nabla^{g,A}\theta = e^f(\mu+i\alpha)(e^1\otimes e^1+e^2\otimes e^2) \otimes\mathfrak l =(\mu+i\alpha)g\otimes c.
\end{equation*}

Finally $\langle\theta,\theta\rangle=c\otimes c$ follows from $\cosh^2(\phi)-\sinh^2(\phi)=1$. We then conclude by Proposition~\ref{prop:complex_Killing_2d}.
\end{proof} 

\begin{thm} 
\label{thm:surfacelocalnormalform}
Let $(N, s, v)$ be a real triple such that $N>|s|$. The following two are equivalent:
\begin{enumerate}[label=\textup{(\roman*)}]
    \item There is a complex generalized Killing spinor with real triple $(N,s,v)$.
    \item Around each point of $M$, there are local coordinates $(x,y)$ such that:
    \begin{gather}
        \dd x =4(N^2-s^2)^{-1}v, \qquad \dd y =4(N^2-s^2)^{-1} * v,\label{eq:dx-dy}\\
        \dd N\wedge \dd x = 0, \qquad \dd s \wedge \dd y = 0,\label{eq:dN-ds}\\
        g = (N^2-s^2)(\dd x\otimes\dd x+\dd y\otimes\dd y).\label{eq:surfacelocalmetric}
    \end{gather}
\end{enumerate}
Furthermore, if the Killing function is $i\mu-\alpha$ and the complex-bilinear square is $\rho=c+\theta$, then, for a unit section $\mathfrak l$ on the complex Hermitian line bundle, we have:
\begin{gather}
    (N^2-s^2) \mu = \partial_x(N),\qquad (N^2-s^2)\alpha = \partial_y(s),\label{eq:derivadas-direccionales}\\
    2\rho = \sqrt{N^2-s^2}\,  (1+N\dd x+i s\dd y )\otimes\mathfrak l, \label{eq:surfacelocalsquare}\\
    \nabla^A\mathfrak l=i\beta\otimes\mathfrak l,\qquad (N^2-s^2)\beta = N\partial_y(s)\dd x+ s\partial_x(N) \dd y.\label{eq:surfacelocalconnection}
\end{gather}
\end{thm}

\begin{proof} 
The differential systems~\eqref{eq:surfacerealsystem} and~\eqref{eq:Lie-dd-v} give: 
\begin{equation*}
\label{eq:surface_closed_coordinates}
\dd \left(\frac{v}{|c|^2}\right) = 0,\qquad
\dd \left(\frac{*v}{|c|^2}\right)=0.
\end{equation*}

These closed one-forms are linearly independent at every point of the regular locus $M^{\circ}$. We choose local primitives $x$ and $y$ such that:
\begin{equation}
\label{eq:dx=v/|c|^2}
\dd x=\frac{v}{|c|^2},\qquad \dd y=\frac{*v}{|c|^2}. 
\end{equation}

The orthonormal coframe~\eqref{eq:surfaceadaptedcoframe} with $v=|c|^2\dd x$ then gives: 
\begin{equation*}
g=4|c|^2(\dd x\otimes\dd x+\dd y\otimes\dd y).
\end{equation*}

The first two equations of~\eqref{eq:surfacerealsystem} imply $\dd N \wedge \dd x =0$ and $\dd s \wedge \dd y =0$ and~\eqref{eq:derivadas-direccionales}. Since $c\neq 0$ on the regular locus $M^\circ$, we choose the local unit section $\mathfrak{l}:=c/|c|$. Then:
\begin{equation*}
c=|c|\mathfrak{l}=\frac{\sqrt{N^2-s^2}}{2}\mathfrak{l}.
\end{equation*}

The identity $(\theta,c)_{\cL}=Nv+i s*v$ gives~\eqref{eq:surfacelocalsquare}; the imaginary part of $\nabla^Ac=(\mu+i\alpha)\theta$ in the unit frame $\mathfrak l$ gives~\eqref{eq:surfacelocalconnection}.\medskip

For the converse, take the trivial complex Hermitian line bundle and the connection given in~\eqref{eq:surfacelocalconnection}, and set:
\begin{equation*}
e^1:=\sqrt{N^2-s^2}\,\dd x,\quad
e^2:=\sqrt{N^2-s^2}\,\dd y,\quad
c:=\frac{\sqrt{N^2-s^2}}2\mathfrak l,\quad
\theta:=\frac12(Ne^1+i s e^2)\otimes\mathfrak l.
\end{equation*}

Then $\langle\theta,\theta\rangle=c\otimes c$. The coframe connection is:
\begin{equation*}
\varpi=\frac{s\partial_y(s)}{N^2-s^2}\dd x+\frac{N\partial_x(N)}{N^2-s^2}\dd y.
\end{equation*}

The identities $N\beta-s\varpi=\partial_y(s)\dd x$ and $N\varpi-s\beta=\partial_x(N)\dd y$ imply:
\begin{equation*}
\nabla^{g,A}\theta
 =\frac{\partial_x(N)+i \partial_y(s)}2(\dd x\otimes e^1+\dd y\otimes e^2)\otimes\mathfrak l
 =(\mu+i\alpha)g\otimes c.
\end{equation*}

as well as:
\begin{align*}
\nabla^Ac&=\frac{\sqrt{N^2-s^2}}2\left(\frac{N\partial_x(N)\dd x-s\partial_y(s)\dd y}{N^2-s^2}+i\beta\right)\otimes\mathfrak l\\
&=(\mu+i\alpha)\frac{\sqrt{N^2-s^2}}2(N\dd x+i s\dd y)\otimes\mathfrak l\\
&=(\mu+i\alpha)\theta.
\end{align*}

Hence, all equations in~\eqref{eq:surfacecovariant} are satisfied and thus we conclude.
\end{proof}


\subsection{Complete Riemannian surfaces and normal forms}


In this subsection, we elaborate on the local characterization obtained in Theorem~\ref{thm:surfacelocalnormalform} to obtain a classification result for generalized Killing spinors on complete Riemannian surfaces.

\begin{prop}\label{prop:surface_killing_tensor}
Let $(N, s, v)$ be the real triple associated with a generalized Killing spinor on $(M,g)$. The symmetric tensor:
\begin{equation}\label{eq:surfacekillingtensor}
N^2g-16v\otimes v=s^2g+16(*v)\otimes(*v)
\end{equation}

is a positive semidefinite quadratic Killing tensor with eigenvalues $s^2,N^2$. If $g$ is complete, then:
\begin{equation}\label{eq:surfaceglobalordering}
|s(p)|\leq N(q)\quad\text{for all }p,q\in M;
\qquad \sup_M|s|\leq\inf_M N.
\end{equation}
\end{prop}

\begin{proof} 
First, we substitute~\eqref{eq:dx=v/|c|^2} with $|c|^2 = (N^2 - s^2)/ 4$ in~\eqref{eq:surfacelocalmetric} and obtain:
\begin{equation*}
(N^2-s^2)g = 16v\otimes v+16(*v)\otimes(*v)
\end{equation*}

and, consequently,~\eqref{eq:surfacekillingtensor}. Furthermore, we use $v((*v)^\sharp)=0$ and $*v(v^\sharp)=0$ to conclude that $K:=N^2g-16 v\otimes v$ is a quadratic form with eigenvalues $s^2$, $N^2$. Now we prove $K$ is Killing. To do so, we compute the covariant derivative and use $\dd N = 4\mu v$ in~\eqref{eq:surfacerealsystem} and $4 \nabla^g v = \mu N g + \alpha s\, \mathrm{vol}_g $ in~\eqref{eq:nabla-v} to obtain:
\begin{align*}
\nabla^gK &= 2N\dd N \otimes g - 16 \nabla^g v\otimes v  - 16 v\otimes \nabla^g v\\
&= 8N\mu  (v\otimes g) - 4 \mu N(g\otimes v + v\otimes g) - 4\alpha s\, (\mathrm{vol}_g \otimes v + v \otimes \mathrm{vol}_g).
\end{align*}

We take the symmetric part and obtain that:
\begin{equation*}
\mathrm{Sym} (\nabla^g K) = 8N\mu  \mathrm{Sym}(v\otimes g) - 4 \mu N \big(\mathrm{Sym}(g \otimes v) + \mathrm{Sym}(v \otimes g)\big)  =0,
\end{equation*}

and thus $K$ is Killing.\medskip

Assume now that $g$ is complete. Let $p,q\in M$ be distinct points in $M$. Since $M$ is connected, the Hopf--Rinow theorem provides a minimizing geodesic joining $p$ to $q$. Parameterizing it by arc length, we obtain:
\begin{equation*}
\gamma:[0,\ell]\longrightarrow M,\qquad
\ell:=d_g(p,q)>0,
\end{equation*}

such that:
\begin{equation*}
\gamma(0)=p,\qquad
\gamma(\ell)=q,\qquad
g_{\gamma(t)}(\dot\gamma(t),\dot\gamma(t))=1.
\end{equation*} 

Since $K$ is Killing, we have that $K_{p}(\dot\gamma(0),\dot\gamma(0))= K_{q}(\dot\gamma(\ell),\dot\gamma(\ell))$. Consequently, we have:
\begin{align*}
s(p)^2 &\leq s(p)^2+
16\bigl((*v)_p(\dot\gamma(0))\bigr)^2 \\
&= K_{p}(\dot\gamma(0),\dot\gamma(0))= K_{q}(\dot\gamma(\ell),\dot\gamma(\ell))
\\
&=
N(q)^2-
16\bigl(v_q(\dot\gamma(\ell))\bigr)^2\leq N(q)^2.
\end{align*}  

Since $N(q)>0$, taking nonnegative square roots finally gives $|s(p)|\leq N(q)$. Note that if $p=q$, the same inequality follows directly from $N(p)^2-s(p)^2=16|v_p|^2\geq 0$. We have therefore proved $|s(p)|\leq N(q)$ for every $p,q\in M$ and, consequently, this proves~\eqref{eq:surfaceglobalordering}. Note that neither the supremum nor the infimum needs to be attained. 
\end{proof}

\begin{lemma}
Let $(N, s, v)$ be the real triple associated with a generalized Killing spinor on $(M,g)$. If $(M,g)$ is complete, then the vector fields:
\begin{equation}\label{eq:surface_complete_fields}
X:=\frac{4v^\sharp}{N}, \, Y:=4(*v)^\sharp \in\Gamma(TM)
\end{equation}
 
are complete,  commute, both vanish precisely on the zero locus of $c\in \Gamma(\cL)$, and their flows preserve $M^\circ$.
\end{lemma}

\begin{proof}
Since a bounded vector field on a complete Riemannian manifold is complete and $|4v^\sharp/N|\leq1$, we have that $4v^\sharp/N$ is complete. On the other hand, by~\eqref{eq:surfaceconformalfields} we have $(*v)^\sharp(N)=0$. Along each integral curve of the second vector field, the speed is at most $N$, which is constant along the curve, so the flow is complete. By~\eqref{eq:surfaceconformalfields}, we compute:
\begin{equation*}
[4N^{-1}v^\sharp,4(*v)^\sharp] = 16N^{-1}[v^\sharp,(*v)^\sharp] -16(*v)^\sharp(N^{-1})v^\sharp=0.
\end{equation*}

Finally, since $N>0$ and $4|v|^2=|c|^2$, both vector fields vanish precisely on $\{c=0\}$. Their complete flows fix $\{c=0\}$ pointwise and are diffeomorphisms, so they preserve its complement.
\end{proof} 

To study the topology of the manifold $(M,g)$, we describe the regular locus $M^\circ$ as a quotient of its universal cover $\widetilde{M^\circ}$. We denote the covering map by $\pi\colon \widetilde{M^\circ} \to M^\circ$ and its deck transformation group by $\Lambda$. In particular, three direct consequences are that: (a) $\Lambda$ is discrete; (b) $M^\circ$ is diffeomorphic to $\widetilde{M^\circ}/\Lambda$; and (c) every object in $M^\circ$ lifts to a $\Lambda$-invariant object in $\widetilde{M^\circ}$. We denote the lifts of the objects on $M^\circ$ by $\widetilde{(\cdot)}$.

\begin{prop}
\label{prop:surfaceregularorbit}
Let $(N, s, v)$ be the real triple associated with a generalized Killing spinor on a complete Riemannian surface $(M,g)$. Then, on the universal cover $\widetilde{M^\circ}$, there are global coordinates $(x,y)\in I\times\mathbb R\subset\mathbb R ^2$ satisfying~\eqref{eq:dx-dy}--\eqref{eq:surfacelocalconnection} for the real triple $(\widetilde{N},\widetilde{s},\widetilde{v})$. Furthermore, the deck transformation group $\Lambda$ acts by translation on those coordinates and is therefore identified with a discrete subgroup of $(\mathbb R ^2, +)$. In particular, if $I\neq\mathbb R$, then $\Lambda$ must be vertical.
\end{prop}

\begin{proof}
On $\widetilde{M^\circ}$, the lifts $\widetilde{X}$, $\widetilde{Y}$ of the two complete commuting fields in~\eqref{eq:surface_complete_fields} are independent. Thus, their flows define an $\mathbb R^2$-action on $ \widetilde{M^\circ}$ whose orbits are two-dimensional. Since every orbit is open, $ \widetilde{M^\circ}$ has a single orbit because it is connected. Therefore, $\widetilde{M^\circ}$ is diffeomorphic to $\mathbb R^2$ and we can choose a coordinate system $(u,y)\in\mathbb R^2$ such that $\partial_u = \widetilde{X}$ and $\partial_y = \widetilde{Y}$ and $\dd u = \widetilde{N} |\widetilde{c}|^{-2} \widetilde{v}$ and $\dd y =  |\widetilde{c}|^{-2} * \widetilde{v}$.\medskip

By $\dd N = 4\mu v$ in~\eqref{eq:surfacerealsystem}, we have $\dd N \wedge \dd u =0$. Then we can change the coordinates by:
\begin{equation*}
x(u):=\int_{u_0}^u \frac{\dd\tau}{\widetilde{N}(\tau)}
\end{equation*}

with $I = x(\mathbb R)$. Therefore, we obtain $\dd x = \widetilde{N}^{-1} \dd u$ and global coordinates $(x,y)\in I\times \mathbb R$. Furthermore, because $\dd x = |\widetilde{c}|^{-2} \widetilde{v}$ and $\dd y =  |\widetilde{c}|^{-2} * \widetilde{v}$ , they coincide with the coordinate system in Theorem~\ref{thm:surfacelocalnormalform} and hence satisfy~\eqref{eq:dx-dy}--\eqref{eq:surfacelocalmetric}  for the real triple $(\widetilde{N},\widetilde{s},\widetilde{v})$. Finally, since $\dd x$ and $\dd y$ are lifts, they are $\Lambda$-invariant. In other words, for every $\gamma\in\Lambda$ we have $\gamma^* \dd x -\dd x =\dd (x \circ \gamma - x) = 0$. Consequently, in these coordinates, $\Lambda$ is identified with a subgroup of $(\mathbb R^2,+)$.
\end{proof}

A direct consequence of Theorem~\ref{thm:surfacelocalnormalform} and Proposition~\ref{prop:surfaceregularorbit} is the following characterization of generalized Killing spinors which are never chiral.

\begin{thm}
\label{thm:3.13}
Let $(M,g)$ be a complete Riemannian surface and let $(N,s,v)$ be a real triple with $N > |s|$. Then the following two are equivalent:
\begin{enumerate}[label=\textup{(\roman*)}]
    \item There exists a complex generalized Killing spinor with real triple $(N,s,v)$.
    \item On $\widetilde{M}$, there exist an open interval $I \subset \mathbb R$ and global coordinates $(x,y) \in I \times \mathbb R $ satisfying~\eqref{eq:dx-dy}--\eqref{eq:surfacelocalmetric} for the real triple $(\widetilde{N},\widetilde{s},\widetilde{v})$.
\end{enumerate}

Furthermore, if $I$ is an open interval, $\Lambda$ is a discrete subgroup acting by translations on $I\times \mathbb R$ and $(\widetilde{N},\widetilde{s},\widetilde{v})$ is a $\Lambda$-invariant real triple satisfying~\eqref{eq:dx-dy}--\eqref{eq:surfacelocalmetric} in the coordinates $(x,y)\in I\times \mathbb R$, then the Riemannian surface $((I\times\R)/\Lambda,g)$ with:
\begin{equation*}
\pi^*g=(\widetilde{N}^2-\widetilde{s}^2)(\dd x\otimes\dd x+\dd y\otimes\dd y),
\end{equation*}

where $\pi$ is the quotient projection, admits a complex generalized Killing spinor.
\end{thm}

Furthermore, if our aim is just to know the existence of nowhere chiral solutions, we have:

\begin{cor} 
On a complete Riemannian surface $(M,g)$, there exists a never chiral complex generalized Killing spinor if and only if, on the universal cover $\widetilde{M}$, there exist global coordinates $(x,y)$ and two $\Lambda$-invariant functions $\widetilde{N}= \widetilde{N}(x)$, $\widetilde{s} = \widetilde{s}(y)$ such that:
\begin{equation*}
\widetilde{g} = (\widetilde{N}(x) ^2 - \widetilde{s}(y)^2)(\dd x\otimes\dd x+\dd y\otimes\dd y).
\end{equation*}
\end{cor}

\begin{proof}
Set $4\widetilde v:= (\widetilde{N} ^2 - \widetilde{s}^2) \dd x$. If we show that $\widetilde v$ is $\Lambda$-invariant, then we apply Theorem~\ref{thm:3.13} and have finished the proof. Let $\gamma\in\Lambda$ be different from the identity. Since $\widetilde g$ is the lift of $g$, $\dd \widetilde N \wedge \dd x = 0$ and $\dd \widetilde s \wedge \dd y = 0$, we have that:
\begin{equation*}
\gamma^*\dd x = \pm \dd x,\qquad \gamma^*\dd y = \pm \dd y.
\end{equation*}

If $\gamma^*\dd x = \dd x$ or $\gamma^*\dd x = - \dd x$, then, using the $\Lambda$-invariance of the Hodge star operator, we have $\gamma^*\dd y =  \dd y$ or $\gamma^*\dd y = -\dd y$, respectively. When $\gamma^*\dd x = - \dd x$, there exist two constants $A$, $B\in\mathbb R$ such that $\gamma(x,y) = (A-x,B-y)$ and hence $\gamma$ has a fixed point. Because $\Lambda$ acts freely, $\gamma$ must be the identity, a contradiction. Hence $\gamma^*\dd x = \dd x$, for all $\gamma\in\Lambda$ and, consequently, $\widetilde v = $ is $\Lambda$-invariant.
\end{proof}

Now we turn our attention to generalized Killing spinors which are chiral at some point, that is, $c$ is somewhere zero. First, we study the number of zeros $c$ can admit and the topology of $M$.
 
\begin{prop}
\label{prop:surfacefilledends}
Let $(M,g)$ be a complete Riemannian surface admitting a generalized Killing spinor with $i\mu-\alpha\not\equiv0$ and square $\rho = c + \theta$, where $c$ has at least one zero. Then precisely one of the following holds:
\begin{equation*}
M\cong\mathbb R^2\text{ with one zero of }c;\qquad
M\cong S^2\text{ with two zeros of }c.
\end{equation*}

In particular, $M^{\circ}$ is diffeomorphic to a cylinder. The metric restricted to $M^{\circ}$ is rotationally symmetric, and either $\alpha\equiv0$ or $\mu\equiv0$. On $M^\circ$ there are oriented coordinates:
\begin{equation}
\label{eq:surfacepolarmetric}
g=\dd r\otimes\dd r+h(r)^2\dd t\otimes\dd t, \qquad t\in\mathbb R/(2\pi\mathbb Z),
\end{equation}

where $0<r<\infty$ for the plane and $0<r<L$ for the sphere. At a zero of $c$, $h$ extends smoothly as an odd function of the geodesic radius, with derivative one in the direction away from that zero.
\end{prop}

\begin{proof}
Each point of the closed discrete zero locus of $c$ has a punctured disk neighborhood and determines a distinct end of $M^\circ$. By Proposition~\ref{prop:surfaceregularorbit}, the regular locus $M^{\circ}$ is topologically a plane, a cylinder, or a torus, with one, two, or zero ends, respectively. Thus, there are at most two zeros and, by the hypothesis, at least one zero, which rules out that $M^\circ$ is topologically a torus.\medskip

First, we prove that $M^\circ$ cannot be diffeomorphic to a plane. If the regular locus were a plane, then $M$ is the one-point compactification of $M^\circ$. In the coordinates $(x,y)$ on this plane, fix $x$ and let $y$ tend to infinity. The resulting points converge to the added point, so continuity implies $N$ is constant. Fixing $y$ and letting $x$ tend to an end point of $I$ similarly makes $s$ constant. Then $16|v|^2=N^2-s^2$ would be a positive constant on the regular locus, contrary to $v=0$ at the added point. We conclude that the regular orbit is a cylinder, and filling one or two ends gives the asserted plane or sphere.\medskip

Since $M^\circ$ is a cylinder, in $(x,y)$-coordinates, the deck transformation group $\Lambda$ is generated by an element $(a_0,b_0)\in \mathbb R^2$. A transverse coordinate on the cylinder is $b_0x-a_0y$. If $a_0\neq0$, then $I=\mathbb R$ and every vertical line escapes to either end. Continuity at a filled end forces $N$ to be constant. Since $\dd N=4\mu v$ in~\eqref{eq:surfacerealsystem} and $v\neq 0$ on the regular locus, $\mu=0$ there and hence everywhere. If $b_0\neq0$, horizontal lines similarly force $s$ constant and $\alpha=0$. Both coordinates of the generator cannot be non-zero, because that would make both invariants constant and prevent $|v|$ from tending to zero.\medskip

If $\alpha=0$, by $2 \mathscr{L}_{(*v)^\sharp}g=- \alpha s g$ in~\eqref{eq:Lie-dd-v}, the field $4(*v)^\sharp$ is Killing and periodic. If $\mu=0$, by $2 \mathscr{L}_{v^\sharp}g = \mu N g$ in~\eqref{eq:Lie-dd-v}, $N$ is constant and $4v^\sharp$ is Killing and periodic. Normalize the resulting effective circle action to have period $2\pi$. It is free on the regular cylinder. At a fixed point, its differential has rotation weight $1$ or $-1$: a higher weight would yield a nonidentity isometry fixing that point with identity derivative, which is impossible on a connected Riemannian manifold. Orthogonal geodesics to the orbits give a global radial coordinate on the orbit interval and the metric~\eqref{eq:surfacepolarmetric}. Smoothness at a fixed point gives the stated parity and derivative conditions. In the one-zero case, completeness forces the other radial end to have infinite geodesic length. In the two-zero case, the interval is $[0,L]$ after adding its endpoints.
\end{proof}

By the previous proposition, the regular locus $M^{\circ}$ of a complete Riemannian surface equipped with a generalized Killing spinor that is chiral at some point is necessarily a cylinder equipped with a rotationally symmetric Riemannian metric. We proceed to establish the polar normal forms of the square of the generalized Killing spinor and its associated real triple.

\begin{thm} 
\label{thm:main-zeros}
Let $(M,g)$ be a complete Riemannian surface and let $(N,s,v)$ be a real triple with $|v| = 0$ somewhere. Then the following two are equivalent:
\begin{enumerate}[label=\textup{(\roman*)}]
    \item There exists a non-parallel generalized Killing spinor with real triple $(N,s,v)$.
    \item There exists a constant $a\neq0$ such that on $M^\circ$ there are positively oriented cylindrical coordinates $(r,t)$ as in~\eqref{eq:surfacepolarmetric} such that one of the following two conditions holds:
    \begin{equation}
    \label{eq:surfacepolar}
    \dd s=0,\quad s\neq0,\quad 4v= ah\dd r \qquad\text{or}\qquad\dd N=0,\quad 4v= ah^2\dd t.
    \end{equation}
\end{enumerate}

Furthermore, if the Killing function is $i\mu-\alpha$ and the complex-bilinear square is $\rho=c+\theta$, then, for a unit section $\mathfrak l$ on the complex Hermitian line bundle, we have:
\begin{equation}
\label{eq:surfacepolar-connection}
\begin{gathered}
\mu=\frac{a\partial_r(h)}{N},\quad \alpha=0,\\
2 \rho=(ah+N\dd r+i s h\dd t)\otimes\mathfrak l\\
\nabla^A\mathfrak l=i\beta\otimes\mathfrak l,\quad \beta=\frac{s\partial_r(h)}{N}\dd t
\end{gathered}\qquad\text{or}\qquad
\begin{gathered}
\mu=0,\quad \alpha=-\frac{\partial_r(s)}{ah},\\
2\rho=\bigl(ah-i s\dd r+Nh\dd t\bigr)\otimes\mathfrak l,\\
\nabla^A\mathfrak l=i\beta\otimes\mathfrak l,\qquad
\beta=\frac{N\alpha}{a}\dd t.
\end{gathered}
\end{equation}
\end{thm} 

\begin{proof}
We first prove (\rm{i}) implies (\rm{ii}). Proposition~\ref{prop:surfacefilledends} gives a regular cylinder with oriented polar metric:
\begin{equation*}
g=\dd r\otimes\dd r+h^2\dd t\otimes\dd t,\qquad
\mathrm{vol}_g =h\dd r\wedge\dd t,\qquad
\ast\dd r=h\dd t,\qquad \ast(h\dd t)=-\dd r,
\end{equation*}

and proves that either $\alpha\equiv0$ or $\mu\equiv0$. These alternatives cannot both hold, by the non-parallel hypothesis. The positively oriented orthonormal coframe $e^1=\dd r$, $e^2=h\dd t$ satisfies:
\begin{equation*}
\nabla^g e^1=\partial_r(h)\dd t\otimes e^2,\qquad
\nabla^g e^2=-\partial_r(h)\dd t\otimes e^1.
\end{equation*}

In particular:
\begin{equation*}
\nabla^g(h\dd r)=\partial_r(h)g,\qquad
\nabla^g(h^2\dd t)=\partial_r(h)\mathrm{vol}_g.
\end{equation*}

If $\alpha\equiv0$ or $\mu\equiv0$, the first two equations in~\eqref{eq:surfacerealsystem} give $\dd s = 0$ or $\dd N =0$, respectively. The effective circle action used in Proposition~\ref{prop:surfacefilledends} is generated by a non-zero constant multiple of $4(\ast v)^\sharp$ or $4 v^\sharp$, respectively. Thus, after choosing its angular parameter of period $2\pi$, there is a constant $a\neq0$ with:
\begin{equation*}
4(\ast v)^\sharp=a\partial_t,\quad 4v=ah\dd r \qquad\text{or}\qquad
4v^\sharp = a\partial_t, \quad 4v= ah^2 \dd t,
\end{equation*}

respectively, which proves (\rm{i}) implies (\rm{ii}). Moreover, we compute the norm of $v$ and obtain, in both cases:
\begin{equation}
\label{eq:normas-v-polar}
N^2 - s^2 = 16|v|^2 = a^2 h^2.
\end{equation}

Furthermore, since we are working in the regular locus $M^\circ$, by Theorem~\ref{thm:surfacelocalnormalform}, there exist local coordinates $(x,y)$ and a unit section of the complex Hermitian bundle $\mathfrak l_{(x,y)}$ such that~\eqref{eq:derivadas-direccionales}--\eqref{eq:surfacelocalconnection} are satisfied. Thus, we substitute~\eqref{eq:surfacepolar} in~\eqref{eq:dx-dy}, use~\eqref{eq:normas-v-polar} and obtain, respectively:
\begin{equation}
\label{eq:relacion-dx-dy-dr-dt}
\dd x = \frac{\dd r}{ah}, \quad \dd y = \frac{\dd t}{a} \qquad\text{or}\qquad
\dd x = \frac{\dd t}{a}, \quad \dd y = \frac{- \dd r}{ah}.
\end{equation}

Finally, we set $\mathfrak l:= \mathrm{sign}(a)\mathfrak l_{(x,y)}$, and substitute~\eqref{eq:relacion-dx-dy-dr-dt} in~\eqref{eq:derivadas-direccionales}--\eqref{eq:surfacelocalconnection} to obtain~\eqref{eq:surfacepolar-connection}.\medskip

Now, we prove (\rm{ii}) implies (\rm{i}). Around every point of $M^\circ$ we can construct a local coordinate system $(x,y)$ defined by the corresponding equations in~\eqref{eq:relacion-dx-dy-dr-dt} and, by the definition of a real triple $(N,s,v)$, we have $N^2 - s^2 = 16|v|^2$. Thus, substituting~\eqref{eq:relacion-dx-dy-dr-dt} in~\eqref{eq:surfacepolarmetric} and~\eqref{eq:surfacepolar}, we obtain~\eqref{eq:dx-dy}--\eqref{eq:surfacelocalmetric}. Consequently, by Theorem~\ref{thm:surfacelocalnormalform}, there exists a complex generalized Killing spinor on $M^\circ$ associated with the data in~\eqref{eq:surfacepolar-connection}. Finally, by Proposition~\ref{prop:reconstruccion-zeros}, which is stated and proved below, any solution satisfying~(\rm{ii}) extends smoothly and uniquely to $M$, giving the existence of a generalized Killing spinor satisfying~(\rm{i}).
\end{proof}

\begin{prop}
\label{prop:reconstruccion-zeros}
Let $M^\circ = (0,L)\times\mathbb R/2\pi\mathbb Z$, with $L\in (0,+\infty]$, be a cylinder with metric:
\begin{equation*}
g=\dd r\otimes\dd r+h(r)^2\dd t\otimes\dd t,\qquad(r,t)\in M^\circ,
\end{equation*}

and we suppose that $h>0$ on the open radial interval and that at every included endpoint $h$ has a smooth odd extension in the distance from that endpoint, with derivative one in the direction away from it. We consider a complex generalized Killing spinor whose real triple $(N,s,v)$ satisfies, respectively, the first or second equation in~\eqref{eq:surfacepolar} and~\eqref{eq:surfacepolar-connection}. Then, collapsing the circle at $r=0$ and, when $L<\infty$, at $r=L$, we obtain a complete Riemannian surface $(M,g)$ to which the complex Hermitian line bundle, its connection and the complex-bilinear square of the spinor extend smoothly and uniquely. In particular, $(M,g)$ admits a complex generalized Killing spinor.
\end{prop} 

\begin{proof} 
We now extend all quantities associated with a generalized Killing spinor with non-trivial chiral locus from $M^{\circ}$ to $M$. We give the extension at $r=0$ first. The smooth odd extension of $h$, together with $\partial_r(h)(0)=1$, implies that $h(r)/r$ extends as a smooth function of $r^2$ with value one at zero. In oriented Cartesian coordinates $x=r\cos(t)$, $y=r\sin(t)$, the metric is:
\begin{equation*}
g=\dd x\otimes\dd x+\dd y\otimes\dd y+\frac{(h(r)/r)^2-1}{r^2}(x\dd y-y\dd x)\otimes(x\dd y-y\dd x).
\end{equation*}

The scalar coefficient on the second line is a smooth function of $r^2$. This shows the metric is smooth at the collapsed circle, with positive-definite value $\dd x\otimes\dd x+\dd y\otimes\dd y$ at its center.\medskip

In the imaginary family ($\dd s =0$ and $s\neq0$), $N=\sqrt{s^2+a^2h^2}$ is smooth in $r^2$ near the zero, with $N(0)=|s|$, and $\mu=a\partial_r(h)/N$ extends smoothly. In the real family ($\dd N =0$), the relation $s^2 = N^2-a^2h^2$ implies that $s$ has a fixed sign near the zero and extends there as that sign times the positive square root, with $s(0)=\pm N$. Thus $s$ is smooth in $r^2$ there. The formula:
\begin{equation*}
\alpha=\frac{a\partial_r(h)}{s}
\end{equation*}

is valid near the zero and proves smoothness of $\alpha$ there; its use is legitimate because $s$ is non-zero on this neighborhood. In either family put $\varepsilon=\operatorname{sign}s(0)\in\{\pm1\}$. The coefficient of $\dd t$ in $\beta$ is a smooth function of $r^2$ with value $\varepsilon$ at zero. Therefore its difference from $\varepsilon$ is $r^2$ times a smooth function of $r^2$.\medskip

Glue a trivial Hermitian line over the polar disk to the line over the regular cylinder by the unitary transition:
\begin{equation*} 
\mathfrak l=e^{i\varepsilon t}\mathfrak l_0,\qquad
\nabla^A\mathfrak l_0
=i(\beta-\varepsilon\dd t)\otimes\mathfrak l_0.
\end{equation*}

The transition is single-valued because $\varepsilon$ is an integer. The connection in the disk frame $\mathfrak l_0$ is smooth: indeed, $r^2\dd t=x\dd y-y\dd x$ is smooth, and the preceding vanishing of its coefficient removes the apparent angular singularity. Set $z_\varepsilon=x+i\varepsilon y=r e^{i\varepsilon t}$. This is only notation for a smooth Cartesian coordinate, with orientation reversed when $\varepsilon=-1$. In the disk frame the scalar section is:
\begin{equation}
\label{eq:surfacezeroscalar}
c=\frac a2\frac{h(r)}r z_\varepsilon\mathfrak l_0,
\end{equation}

and it is smooth. If $p$ is a zero of $c$, then:
\begin{equation*}
(\nabla^A c)_p=\frac a2(\dd x+i\varepsilon\dd y)_p\otimes(\mathfrak l_0)_p.
\end{equation*}

Since $c(p)=0$, this derivative is independent of the choice of connection. As a real-linear map $T_pM\to\cL_p$ it is an isomorphism, so the zero of $c$ is simple.\medskip

For the one-form component, the identities:
\begin{equation*}
e^{i\varepsilon t}(\dd r+i\varepsilon r\dd t)=\dd z_\varepsilon,\qquad
r^2e^{i\varepsilon t}(\dd r-i\varepsilon r\dd t)
=z_\varepsilon^2\dd\overline{z}_\varepsilon
\end{equation*}

give the following expressions on the punctured disk. In the imaginary family:
\begin{equation*}
\theta=\left(\frac{N+\varepsilon s  h/r}{4}\dd z_\varepsilon+\frac{N-\varepsilon s  h/r}{4r^2}z_\varepsilon^2\dd\overline{z}_\varepsilon\right)\otimes\mathfrak l_0.
\end{equation*}

Both scalar coefficients extend as smooth functions of $r^2$: the numerator of the second vanishes at $r=0$, since $N(0)=\varepsilon s =| s |$ and $h(r)/r\to1$. Thus, its quotient by $r^2$ is smooth. The value of $\theta$ at the zero is $| s |\dd z_\varepsilon\otimes\mathfrak l_0/2$, which is non-zero. In the real family, the corresponding expression is:
\begin{equation*}
\theta=\left(-\frac{i(s+\varepsilon  N  h/r)}4\dd z_\varepsilon-\frac{i(s-\varepsilon  N  h/r)}{4r^2}z_\varepsilon^2\dd\overline{z}_\varepsilon\right)\otimes\mathfrak l_0.
\end{equation*}

Here $s(0)=\varepsilon  N  $, so the same argument proves smoothness, and the value at zero is:
\begin{equation*}
-i\varepsilon  N  \dd z_\varepsilon\otimes\mathfrak l_0/2\neq0.
\end{equation*}

Thus all line-valued tensors extend smoothly, and $\theta$ remains nowhere vanishing. Since both sides of each equation in~\eqref{eq:surfacecovariant} are smooth and agree off the zero, they agree at the zero as well.\medskip

If a second zero is present, use the positively oriented polar coordinates $(L-r,-t)$ near it. Replacing $a$ and $\mathfrak l$ by $-a$ and $-\mathfrak l$ leaves $c+\theta$, the connection, and the coefficient functions unchanged when expressed in the new coordinates. The same construction therefore applies at this zero, with $\varepsilon$ now equal to the sign of its limiting value of $s$. Take the two polar disks disjoint. Their unitary transition maps with the regular cylinder define a global Hermitian line bundle and the global square and connection just constructed.\medskip

For a solution originally given on $M$, the preceding reconstruction agrees with it on $M^\circ$ under the unitary isomorphism determined by $\mathfrak l=2c/(ah)$. This isomorphism also extends over each zero of $c$. To see this without assuming any regularity of the singular cylinder frame, choose a smooth real vector field $X$ non-zero at that zero. The constructed one-form evaluates non-trivially on $X$ there, by the displayed Cartesian formulas. The original one-form does so as well: at $c=0$,~\eqref{eq:surfacehermitiantensor} gives:
\begin{equation*}
|\theta(X)|_{\cL}^2=\frac{N^2}{4}g(X,X)>0.
\end{equation*}

After shrinking the disk, these evaluations are nowhere vanishing sections of the two line bundles. The unique complex-linear map sending the constructed evaluation to the original one defines a smooth line-bundle isomorphism on the disk. On the punctured disk it is the already prescribed isomorphism. By continuity, it is unitary at the center, and it identifies $c,\theta$ and their connections there because the corresponding smooth tensors agree on the punctured disk. This also proves uniqueness, and applies to any two smooth extensions of the same regular data.\medskip

The sphere metric is complete by compactness. For the plane, the extended radial function has range $[0,\infty)$, and the polar metric implies that the length of any piecewise smooth curve is at least the absolute variation of its radial coordinate. Consequently, a closed metric ball about the zero is contained in a set $\{r\leq b\}$ for some finite $b$. This set is compact: it is the image of $[0,b]\times S^1$ with the circle at zero collapsed. The closed ball is therefore compact, and the Hopf--Rinow theorem proves completeness. At $r=0$, the non-zero coefficients are respectively:
\begin{equation*}
\mu(0)=\frac{a}{| s |}\neq0,\qquad
\alpha(0)=\frac{a}{s(0)}\neq0.
\end{equation*}

The same conclusion holds at a second zero after the coordinate change above. Thus every polar reconstruction is non-parallel. Proposition~\ref{prop:complex_Killing_2d} then gives the stated spin-c existence interpretation, without prescribing a spin-c lift in advance. 
\end{proof} 

\begin{cor}
Every complete Riemannian surface admitting a complex generalized Killing spinor with $i\mu-\alpha\not\equiv0$ is diffeomorphic to $\mathbb R^2$, $S^1\times\mathbb R$, $T^2$, or $S^2$.
\end{cor}

\begin{proof}
In the regular case, a discrete translation subgroup has rank zero, one, or two. If $I\neq\mathbb R$, Proposition~\ref{prop:surfaceregularorbit} implies its elements are vertical, so the quotient is a plane or a cylinder. If $I=\mathbb R$, the three possible ranks give a plane, a cylinder, or a torus. The polar case gives a plane or a sphere by Proposition~\ref{prop:surfacefilledends}.
\end{proof}
 
\begin{remark}
The previous construction in Proposition~\ref{prop:reconstruccion-zeros} determines the topology of $\cL$. Give its fibers their complex orientation and $M$ its prescribed orientation. In the oriented Cartesian coordinates used in the proof,~\eqref{eq:surfacezeroscalar} has local complex coefficient $a h(r)e^{i\varepsilon t}/2$. On a sufficiently small positively oriented circle around the zero, its normalization is:
\begin{equation*}
t\mapsto\frac{a}{|a|}e^{i\varepsilon t}.
\end{equation*}

Its winding number is $\varepsilon$. Hence the local index of the section at every zero $p$ is:
\begin{equation*}
\operatorname{ind}_p(c)=\operatorname{sign}s(p).
\end{equation*}

At a second zero, the same calculation uses the oriented coordinates $(L-r,-t)$; replacing $a$ with $-a$ changes only the constant unit factor, not the index. On a sphere, the first Chern number of a complex line bundle is the sum of the indices of the isolated zeros of a section. Since these two zeros are all the zeros of $c$, it follows that:
\begin{equation*}
\deg\cL:=\langle c_1(\cL),[M]\rangle
=\operatorname{sign}s(p_0)+\operatorname{sign}s(p_1)
\in\{-2,0,2\}.
\end{equation*}

In the imaginary family the two signs agree, so $\deg\cL=2\operatorname{sign} s $. In the real family, the degree is zero precisely when the values of $s$ at the zeros have opposite signs. In particular, all these degrees are even, as required by $c_1(\cL)\bmod2=w_2(TM)=0$ for a characteristic line on the sphere. On the plane, every complex line bundle is trivial because the base is contractible. In the regular case, the nowhere vanishing section $c$ trivializes $\cL$, so the characteristic-line condition is satisfied there as well. Triviality of the line bundle does not imply flatness of the specified unitary connection.
\end{remark}

\begin{cor}
\label{cor:surface_complete_separation}
For a generalized Killing spinor with $i\mu-\alpha\not\equiv0$ on a complete Riemannian surface, we have:
\begin{equation*}
\alpha\mu\equiv0\quad\Longleftrightarrow\quad
\alpha\equiv0\ \text{or}\ \mu\equiv0.
\end{equation*}

If neither coefficient vanishes identically, then $c$ is nowhere zero, $\cL$ is trivial, and both one-variable functions $N,s$ in the regular normal form are non-constant.
\end{cor}

\begin{proof}
The reverse implication is immediate. Suppose that $\alpha\mu\equiv0$. If $c$ has a zero, Proposition~\ref{prop:surfacefilledends} implies that either $\alpha\equiv0$ or $\mu\equiv0$. Assume therefore that $c$ is nowhere vanishing. By Theorem~\ref{thm:surfacelocalnormalform} and Proposition~\ref{prop:surfaceregularorbit}, on the global regular normal form we have:
\begin{equation*}
(N^2-s^2)\mu= \partial_x(N),\qquad (N^2-s^2)\alpha= \partial_y(s).
\end{equation*}

Since $N^2-s^2>0$, the condition $\alpha\mu\equiv0$ gives $\partial_x(N) \partial_y(s) = 0$ for every $x,y$. Hence either $\partial_x(N)\equiv0$ or $\partial_y(s)\equiv0$, and therefore either $\mu\equiv0$ or $\alpha\equiv0$.\medskip

Finally, if neither $\alpha$ nor $\mu$ vanishes identically, the first part shows that $c$ cannot have zeros. Thus $c$ trivializes $\cL$, and the equations above imply that both $N$ and $s$ are non-constant.
\end{proof}

\begin{remark}\label{rem:surface_mixing_completeness}
Completeness does \emph{not} imply $\alpha\mu\equiv0$. It excludes only disjoint-support mixing, meaning $\alpha\mu\equiv0$ with neither coefficient identically zero. The examples include complete regular solutions with $\alpha\mu>0$ everywhere. Topological triviality of the characteristic line does not imply flatness of its connection.
\end{remark}

\begin{cor}
\label{cor:surfacecompactclassification}
An oriented compact Riemannian surface admits a generalized Killing spinor with $i\mu-\alpha\not\equiv0$ if and only if one of the following holds:
\begin{enumerate}[label=\textup{(\roman*)}]
\item It is a rotationally symmetric sphere. In particular, $\alpha\equiv0$ or $\mu\equiv0$.
\item It is a torus with a global presentation $(M,g)=(\R^2/\Lambda,g)$ with:
\begin{equation*}
\pi^*g=(\widetilde{N}(x)^2-\widetilde{s}(y)^2)(\dd x\otimes\dd x+\dd y\otimes\dd y),
\end{equation*}
where $\pi$ is the quotient projection, $\Lambda$ is a full lattice, and $\widetilde{N}$, $\widetilde{s}$ are two $\Lambda$-invariant functions with $\widetilde{N}>|\widetilde{s}|$, and at least one of $\widetilde{N}$, $\widetilde{s}$ is non-constant.
\end{enumerate}

Furthermore, in {\rm(i)}, every such solution has two simple zeros of $c$ and belongs to one of the families in Theorem~\ref{thm:main-zeros}. In {\rm(ii)}, $c$ is nowhere zero and $\cL$ is trivial.
\end{cor}

Consequently, no closed surface of genus at least two admits a solution with $i\mu-\alpha\not\equiv0$. Genuine mixing of $\mu$ and $\alpha$ on a closed surface occurs only in the toroidal family.

\begin{cor}
\label{cor:surfacetorusconsequences}
Every toroidal solution in Corollary~\ref{cor:surfacecompactclassification} has a point where $\mu + i\alpha=0$. 
\end{cor}

\begin{proof}
For a toroidal solution, $c$ is nowhere zero. Since $4|v|^2=|c|^2$, it follows that $v$ is nowhere vanishing. By compactness, the function $N+s$ attains a maximum at some point $p\in M$. The real differential system gives:
\begin{equation*}
0=\dd(N+s)_p=4\mu(p)v_p+4\alpha(p)(*v)_p.
\end{equation*}

Since the one-forms $v_p$ and $(*v)_p$ are non-zero and orthogonal, they are
linearly independent. Hence, $\mu(p)=\alpha(p)=0$.
\end{proof}


\subsection{Explicit examples} 


Because of Proposition~\ref{prop:complex_Killing_2d}, we have that $\eta\in \Gamma(S)$ is a complex generalized spin Killing spinor with complex Killing function $i\mu-\alpha$ and $\rho = c + \theta$ is its complex-bilinear square if and only if:
\begin{equation}\label{eq:spin-Kill}
\dd c (X) = (i\alpha+\mu)\theta(X), \qquad \nabla^g _X \theta = (i\alpha + \mu)c X^{\flat}.
\end{equation}

\begin{prop}\label{prop:spin-plano}
Let $M\subset\C$ be a connected open set endowed with the Euclidean metric. Then there exists a complex generalized spin Killing spinor $\eta\in\Gamma(S)$ with complex Killing function $i\mu-\alpha\not\equiv0$ such that $\rho = c + \theta$ is its complex-bilinear square if and only if there exist two functions $F(z):=-az+b_1$ and $G(z):=a\overline z+b_2$ for some constants $a$, $b_1$, $b_2\in\mathbb{C}$ with $a\neq0$ such that:
\begin{equation*} 
\theta=i(F-G)\dd x+(F+G)\dd y, \qquad c^2=4\,FG, \qquad(i\alpha + \mu) c = -2 i \,a.
\end{equation*}
\end{prop}

\begin{proof}
We can write $\theta = u \dd x + v \dd y$ and the second equation in~\eqref{eq:spin-Kill}, applied to $X= \frac{\partial}{\partial x}$ and $X= \frac{\partial}{\partial y}$, respectively, becomes:
\begin{equation*}
\frac{\partial u}{\partial x} \dd x + \frac{\partial v}{\partial x} \dd y = (i\alpha + \mu)c \dd x,\qquad \frac{\partial u}{\partial y} \dd x + \frac{\partial v}{\partial y} \dd y = (i\alpha + \mu)c \dd y.
\end{equation*}

In other words, $\dd u = (i\alpha + \mu)c \, \dd x$, and $\dd v = (i\alpha + \mu)c \, \dd y$. We define:
\begin{equation*}
F:=\frac{v-iu}{2}, \qquad G:=\frac{v+iu}{2}.
\end{equation*}

Differentiating, we obtain that:
\begin{equation*}
\dd F = - \frac{i(i\alpha + \mu)c}{2} \dd z, \qquad \dd G = \frac{i(i\alpha + \mu)c}{2} \dd \overline{z}.
\end{equation*}

Thus, $F$ is holomorphic, $G$ is anti-holomorphic, and:
\begin{equation*}
-\frac{\partial F}{\partial x} = \frac{\partial G}{\partial x} = \frac{i(i\alpha + \mu)c}{2}, \qquad \frac{\partial F}{\partial y} = \frac{\partial G}{\partial y} = \frac{(i\alpha + \mu)c}{2}.
\end{equation*}

Thus, $\frac{i(i\alpha + \mu)c}{2}$ is holomorphic and anti-holomorphic and, consequently, it must be equal to a constant $a\in\mathbb{C}$. Moreover, $a\neq0$. Indeed, if $a=0$, then $(i\alpha+\mu)c=0$. Since $i\alpha+\mu\not\equiv0$, the open set $U:=\{p\in M\mid (i\alpha+\mu)(p)\neq0\}$ is nonempty. On $U$ we have $c=0$, hence $\dd c=0$. The first equation in~\eqref{eq:spin-Kill} then gives $0=\dd c=(i\alpha+\mu)\theta$ on $U$, contradicting the fact that $\theta$ is nowhere vanishing. Finally, there exist two constants $b_1$ and $b_2\in\mathbb{C}$ such that:
\begin{equation*}
F(z) = -az + b_1, \qquad G(z) = a \overline z + b_2.
\end{equation*}

Moreover:
\begin{equation*}
c^2=\escal{\theta,\theta}=- (F-G)^2+(F+G)^2=4FG.
\end{equation*}

Conversely, since $a\neq0$ and $(i\alpha + \mu) c = -2 i \,a$, both $c$ and $i\alpha+\mu$ are nowhere vanishing. We differentiate $c^2 = 4 \,FG$ and obtain:
\begin{align*}
\dd c &= \frac{2}{c} ( F \dd G +  G \dd F) \\
    &= \frac{2}{c} \left(F\frac{i(i\alpha + \mu)c}{2} \dd \overline{z} - G \frac{i(i\alpha + \mu)c}{2} \dd z  \right)\\
    &= i(i\alpha + \mu) \left(F \dd x - F i\dd y - G\dd x -Gi\dd y \right) \\
    &= (i\alpha + \mu) \left(i(F-G) \dd x + (F +G) \dd y\right) \\
    &= (i\alpha + \mu) \theta,
\end{align*}

which is the first identity in Equation~\eqref{eq:spin-Kill}. Analogously, taking the $\nabla^g$-derivative of $\theta$, we obtain:
\begin{align*}
\nabla^g\theta &= i (\dd F - \dd G)\otimes \dd x + (\dd F + \dd G)\otimes \dd y \\
    &= i \left(- \frac{i(i\alpha + \mu)c}{2}\right) (\dd z + \dd \overline{z})\otimes \dd x + \left(- \frac{i(i\alpha + \mu)c}{2}\right) (\dd z - \dd \overline{z}) \otimes \dd y \\
    &= (i\alpha + \mu)c g,
\end{align*}

which is the second identity in Equation~\eqref{eq:spin-Kill}. Together with $\escal{\theta,\theta}=c^2$, Proposition~\ref{prop:complex_Killing_2d} gives the required generalized spin Killing spinor.
\end{proof}

\begin{example}\label{ex:spin-counterexample}
Let $M= \{z=x+iy\in \mathbb{C}\mid |z|< \tfrac{1}{2}\}$ and equip $M$ with the Euclidean metric $g$ and the standard spin structure with complex spinor bundle $S$. Let $\sqrt{z}$ be the principal branch of the square root on $\mathbb{C} \setminus\{x\leq0,\,y=0\}$. We define:
\begin{equation*}
F(z) := 1-z, \qquad G(z) :=  \overline z + i,\qquad c := 2\sqrt{1-z}\sqrt{\overline{z}+i}.
\end{equation*}

Therefore, by Proposition~\ref{prop:spin-plano}, there exists a complex generalized spin Killing spinor with complex Killing function:
\begin{equation*}
-\alpha(z) + i\mu(z) = \frac{1}{\sqrt{1-z}\sqrt{\overline{z}+i}}.
\end{equation*}

We now evaluate at $z=0$ and obtain:
\begin{equation*}
-\alpha(0)+i\mu(0)=\frac{1}{\sqrt{i}}=\frac{1-i}{\sqrt{2}}
\end{equation*}

and $\alpha(0)\mu(0)=\frac{1}{2}\neq0$. This example is a proper complex generalized spin Killing spinor with complex Killing function $i\mu-\alpha$ with $\alpha\mu\not\equiv 0$ and, therefore, contradicts~\cite[Lemma~3.4]{Grosse_Nakad_2015} in dimension two.
\end{example}

\begin{example}\label{ex:nowhere-vanishing}
Let us consider the following oriented and connected two-dimensional manifold:
\begin{equation*}
M:=\{(x,y)\in\R^2\mid y<x\}\subset\R^2.
\end{equation*}

Take $t:=e^{y-x}$, so $0<t<1$, and define $\phi:=\operatorname{arctanh}(t)$, so that:
\begin{equation*}
\cosh(\phi)=\frac{1}{\sqrt{1-t^2}},\qquad \sinh(\phi)=\frac{t}{\sqrt{1-t^2}}.
\end{equation*}

Set:
\begin{equation*}
e^1:=\sqrt{1-t^2}\dd x,\qquad e^2:=\frac{\sqrt{1-t^2}}{t}\dd y
\end{equation*}

and equip $M$ with the Riemannian metric $g:=e^1\otimes e^1+e^2\otimes e^2$. Finally, define:
\begin{equation*}
\alpha:=\frac{t^2}{1-t^2},\qquad \mu:=\frac{1}{1-t^2}.
\end{equation*}

In particular, they satisfy that $\alpha>0$ and $\mu>0$ on $M$. We can explicitly check all the conditions in Proposition~\ref{prop:complex_Killing_2d_frame}. First of all:
\begin{equation*}
\mu\cosh(\phi)e^1-\alpha\sinh(\phi)e^2=\frac{1}{1-t^2}(\dd x-t^2\dd y)=\dd(x+\tfrac{1}{2}\log(1-t^2)),
\end{equation*}

where we have used that $\dd t=t(\dd y-\dd x)$. Hence, the above one-form is exact. Now, using that $e^1\wedge e^2=\frac{1-t^2}{t}\dd x\wedge \dd y$ we obtain:
\begin{align*}
\dd e^1&=\frac{t^2}{\sqrt{1-t^2}}\dd x\wedge \dd y=\alpha\sinh(\phi)e^1\wedge e^2,\\
\dd e^2&=\frac{1}{t\sqrt{1-t^2}}\dd x\wedge\dd y=\mu\cosh(\phi)e^1\wedge e^2.
\end{align*}

Finally, the function $\phi=\operatorname{arctanh}(t)$ satisfies:
\begin{equation*}
\dd\phi=\frac{1}{1-t^2}\dd t=\frac{t}{1-t^2}(\dd y-\dd x)=-\mu\sinh(\phi)e^1+\alpha\cosh(\phi)e^2.
\end{equation*}

Therefore, we have constructed an example of a Riemannian surface carrying a generalized spin-c Killing spinor with complex Killing function $i\mu-\alpha$ for which $\alpha\mu>0$.
\end{example}

\begin{example}\label{ex:alpha.mu=0-2}
Let us consider the following oriented and connected two-dimensional manifold:
\begin{equation*}
M:=\{(x,y)\in\R^2\mid x>0\text{ or }y<0\}\subset\R^2.
\end{equation*}

Choose non-constant smooth functions $F,G\colon\R\to[0,\infty)$ such that $F(x)=0$ for $x\geq0$ and $G(y)=0$ for $y\leq0$. Now fix a positive constant $\phi_0>0$ and define:
\begin{align*}
\phi(x,y)&:=\phi_0+F(x)+G(y)>0,\\
f(x,y)&:=-\log\cosh(\phi_0+F(x))-\log\sinh(\phi_0+G(y)).
\end{align*}

Set:
\begin{equation*}
e^1:=\frac{1}{\cosh(\phi)}\dd y,\qquad e^2:=-\frac{1}{\sinh(\phi)}\dd x
\end{equation*}

and equip $M$ with the Riemannian metric $g=e^1\otimes e^1+e^2\otimes e^2$. Finally, define:
\begin{equation*}
\alpha:=-F'(x)\tanh(\phi),\qquad \mu:=-\frac{G'(y)}{\tanh(\phi)}.
\end{equation*}

If $F'(x)\neq0$, then $x<0$. But on $M$, $x<0$ implies that $y<0$, thus $G'(y)=0$. Similarly, if $G'(y)\neq0$, then $y>0$, which implies that $x>0$, thus $F'(x)=0$. Therefore $\alpha\mu\equiv0$ everywhere on $M$, while $\alpha\not\equiv0$ and $\mu\not\equiv0$. A direct computation shows that all the equations in Proposition~\ref{prop:complex_Killing_2d_frame} are satisfied. Therefore, we have constructed an example of a Riemannian surface carrying a generalized spin-c Killing spinor with complex Killing function $i\mu-\alpha$ for which $\alpha\mu\equiv0$ but neither $\alpha$ nor $\mu$ is identically zero.
\end{example}

\begin{example}\label{example:complete_2d}
Consider $\R^2$ and choose constants $a,b>0$ and $C>a+b$. Define:
\begin{equation*}
A(y):=a\tanh(y),\qquad B(x):=C+b\tanh(x),\qquad t(x,y):=e^{A(y)-B(x)}
\end{equation*}

and:
\begin{equation*}
f(x,y):=B(x)+\frac{1}{2}\log(1-t(x,y)^2).
\end{equation*}

Since $-a<A(y)<a$ and $C-b<B(x)<C+b$, we have $0<m\leq t(x,y)\leq M<1$ for some constants $m,M$. In particular, the function $\phi:=\operatorname{arctanh}(t)>0$ is globally defined. Set:
\begin{equation*}
e^1:=\sqrt{1-t^2}\dd x,\qquad e^2:=\frac{\sqrt{1-t^2}}{t}\dd y
\end{equation*}

and equip $\R^2$ with the Riemannian metric $g:=e^1\otimes e^1+e^2\otimes e^2$. Finally, define:
\begin{equation*}
\alpha:=A'(y)\frac{t^2}{1-t^2},\qquad \mu:=B'(x)\frac{1}{1-t^2}.
\end{equation*}

Since $A'(y)=\frac{a}{\cosh^2(y)}>0$ and $B'(x)=\frac{b}{\cosh^2(x)}>0$, we have that $\alpha>0$, $\mu>0$, and $\alpha\mu>0$. A direct computation shows that all the equations in Proposition~\ref{prop:complex_Killing_2d_frame} are satisfied. We now show completeness of the metric $g$. Since:
\begin{equation*}
g=e^1\otimes e^1+e^2\otimes e^2=(1-t^2)\dd x\otimes\dd x+\frac{1-t^2}{t^2}\dd y\otimes\dd y
\end{equation*}

and $0<m\leq t\leq M<1$ we have:
\begin{equation*}
0<1-M^2\leq1-t^2\leq1-m^2<1,\qquad 0<\frac{1-M^2}{M^2}\leq\frac{1-t^2}{t^2}\leq\frac{1-m^2}{m^2}<\infty.
\end{equation*}

From these inequalities we conclude that there exists a real constant $K\geq1$ such that:
\begin{equation*}
\frac{1}{K}g_0\leq g\leq Kg_0,
\end{equation*}

where $g_0:=\dd x\otimes\dd x+\dd y\otimes\dd y$ is the Euclidean metric on $\R^2$. Hence, the metric $g$ is uniformly equivalent to the Euclidean metric $g_0$, which is complete. Therefore, the metric $g$ is also complete.
\end{example}

\begin{example}\label{ex:compact}
Let us consider the same data as in Example~\ref{example:complete_2d} but taking instead:
\begin{equation*}
A(y):=a\sin(y),\qquad B(x):=C+b\sin(x),
\end{equation*}

which are $2\pi$-periodic functions. Then the whole construction descends to the torus $T^2=\R^2/(2\pi\Z)^2$, giving us an example of a compact Riemannian surface carrying a generalized spin-c Killing spinor with complex Killing function $i\mu-\alpha$ such that:
\begin{equation*}
\alpha\mu=ab\cos(x)\cos(y)\frac{t^2}{(1-t^2)^2}.
\end{equation*}

Note that this expression is non-zero on an open dense set and takes both signs. At points where $\cos(x)=\cos(y)=0$, both coefficients vanish, in agreement with Corollary~\ref{cor:surfacetorusconsequences}.
\end{example}


\section{Generalized Killing spinors and trans-Sasakian three-manifolds}
\label{sec:complex_Killing}


In this section, we consider complex generalized spin-c Killing spinors on three-manifolds and establish an equivalence between manifolds carrying such spinors and $\mu$-exact trans-Sasakian manifolds of type $(\alpha,\mu)$, analogously to \cite{Sasakian_spin-c_2026,Kenmotsu_spin-c_2026}. We also provide a counterexample to the Gro\ss e--Nakad conjecture and prove that the conjecture holds when the characteristic vector field of the trans-Sasakian structure is complete.\medskip
   
First, we characterize the complex-bilinear square of a complex generalized spin-c Killing spinor in this dimension.

\begin{prop}
Let $(M,g)$ be a three-dimensional Riemannian manifold and let $\mathcal{L}$ be a complex line bundle. An $\mathcal{L}$-valued complex exterior form $\rho\in\Omega_\C(M,\mathcal{L})$ is the complex-bilinear square of a nowhere vanishing irreducible complex spinor $\eta\in\Gamma(S)$ for a certain bundle of irreducible complex spinors $S$ on $(M,g)$ if and only if $\rho\in\Omega^1_\C(M,\mathcal{L})$ is nowhere vanishing and isotropic.
\end{prop}

\begin{proof}
The complex-bilinear square of $\eta$ was computed pointwise in~\cite[Cor.\ 5.5]{Algebraic_Complex_Square_2025} using a skew-symmetric admissible complex-bilinear pairing of negative adjoint type. The globalization to an $\mathcal{L}$-valued complex exterior form, as explained in~\cite{Complex_Differential_Spinors_2026}, yields the result of the statement.
\end{proof}

\begin{prop}\label{prop:complex_Killing}
A three-dimensional spin-c Riemannian manifold $(M,g)$ admits a complex generalized spin-c Killing spinor $\eta\in\Gamma(S)$ with complex Killing function $i\mu-\alpha$, where $\alpha$, $\mu\in C^\infty(M,\R)$, if and only if:
\begin{equation*}
\nabla^{g,A}_X\rho=(i\mu-\alpha)*(X^\flat\wedge\rho)
\end{equation*}

for all $X\in\Gamma(TM)$, where $\rho$ denotes the complex-bilinear square of $\eta$.
\end{prop}

\begin{proof}
Recall that we have equipped the bundle of irreducible complex spinors $S$ with a skew-symmetric admissible bilinear pairing $\scB$ of negative adjoint type. As explained in~\cite{LBC13,LBC16}, and further elaborated in~\cite{Algebraic_Complex_Square_2025}, the truncated Kähler-Atiyah bundle is defined on $\wedge^<T^*_\C M=\C\oplus T^*_\C M$ endowed with the product:
\begin{equation*}
\rho_1\vee\rho_2=\mathcal{P}_<(\rho_1\diamond\rho_2-i*\tau(\rho_1\diamond\rho_2))
\end{equation*}

for all $\rho_1$, $\rho_2\in\wedge^<T^*_\C M$. Here $\mathcal{P}_<\colon\wedge T^*_\C M\to\wedge^<T^*_\C M$ is the natural projection, $\diamond$ denotes the geometric product, and $\tau$ is the reversion anti-automorphism that acts as multiplication by $\smash{(-1)^{\binom{k}{2}}}$ on $k$-forms. By~\cite[Thm.\ 5.10]{Complex_Differential_Spinors_2026}, the equation $\nabla^{g,A}_X\eta=\frac{i\alpha+\mu}{2}X\cdot\eta$ is equivalent to:
\begin{equation*}
\nabla^{g,A}_X\rho=\frac{i\alpha+\mu}{2}(X^\flat\vee\rho-\rho\vee X^\flat)=\frac{i\alpha+\mu}{2}(2i*(X^\flat\wedge\rho))=(i\mu-\alpha)*(X^\flat\wedge\rho),
\end{equation*}

which proves the statement.
\end{proof}


\subsection{Equivalence with trans-Sasakian structures}
\label{sec:equivalence_trans-Sasakian}


Let $M$ be an oriented three-manifold. Recall that an almost contact metric structure on $M$ is a tuple $(g,\varphi,\xi,\zeta)$ consisting of a Riemannian metric $g$, a nowhere vanishing vector field $\xi\in\Gamma(TM)$, its metric dual one-form $\zeta \in \Omega^1(M)$, and an endomorphism $\varphi\in\Gamma(\End(TM))$, satisfying the following equations:
\begin{equation}\label{eq:metricontact}
\begin{gathered}
\varphi(\xi) = 0, \qquad \zeta(\varphi(X)) = 0, \qquad \varphi^2 (X) = -X + \zeta(X) \xi, \\
g(\varphi(X), Y) = -g(X, \varphi (Y) ), \qquad g(\varphi (X), \varphi (Y)) = g(X, Y) - \zeta(X)\zeta(Y)
\end{gathered}
\end{equation}

for every $X$, $Y\in\Gamma(TM)$. We will denote by $\cD := \mathrm{Ker}(\zeta) \subset TM$ the real distribution defined as the kernel of the almost contact one-form $\zeta\in \Omega^1(M)$. An almost contact metric structure is then said to be \emph{trans-Sasakian} if there exist functions $\alpha$ and $\mu$ on $M$ such that the following equation is satisfied:
\begin{equation}
\label{eq:trans}
(\nabla^g_X \varphi)Y = \alpha (g(X, Y)\xi - \zeta(Y)X) + \mu (g(\varphi (X), Y)\xi - \zeta(Y)\varphi (X)).
\end{equation}

Given real functions $\alpha$ and $\mu$, if this is the case, we say that $(g,\varphi,\xi,\zeta)$ is of type $(\alpha,\mu)$.


\subsubsection{Proof of the only if direction in Theorem~\ref{thm:ifandonlyif}}


Assume that the three-dimensional oriented Riemannian manifold $(M,g)$ admits a complex generalized spin-c Killing spinor $\eta\in\Gamma(S)$ with complex Killing function $i\mu-\alpha$, where $\alpha$, $\mu\in C^\infty(M,\R)$. Let $\rho\in\Omega^1_\C(M,\mathcal{L})$ be its corresponding complex-bilinear square. We will construct a trans-Sasakian structure $(g,\varphi,\xi,\zeta)$ of type $(\alpha,\mu)$ on $(M,g)$ directly from $\rho$.\medskip

First, we construct the almost contact one-form, whose metric dual we will identify with the characteristic vector field of the trans-Sasakian structure. Taking the exterior product of $\rho$ with its complex conjugate $\overline{\rho}$, we see that the line bundle factors trivialize under the canonical pairing $\mathcal{L}\otimes\overline{\mathcal{L}}\cong\mathcal{L}\otimes\mathcal{L}^*\cong\C$, yielding a globally well-defined real two-form $i\rho\wedge\overline{\rho}\in\Omega^2(M)$. Applying the Hodge star operator, we define the one-form $\zeta\in\Omega^1(M)$ by:
\begin{equation*}
\zeta:=\frac{i*(\rho\wedge\overline{\rho})}{\escal{\rho,\overline{\rho}}},
\end{equation*}

where $\escal{\rho,\overline{\rho}}\in C^\infty(M,\R)$ is the inner product defined by $g$. The dual $\xi:=\zeta^\sharp$ of $\zeta$ is a nowhere vanishing vector field on $M$ of unit length and $\zeta(\xi)=1$. We note that the fundamental two-form $\omega\in\Omega^2(M)$ in dimension three is given by:
\begin{equation*}
\omega:=-i\frac{\rho\wedge\overline{\rho}}{\escal{\rho,\overline{\rho}}}=-*\zeta.
\end{equation*}

\begin{lemma}\label{lemma:nabla_zeta}
The one-form $\zeta\in\Omega^1(M)$ constructed above from the complex-bilinear square $\rho$ satisfies:
\begin{equation*}
\nabla^g_X\zeta=-\alpha\iota_X\omega+\mu(X^\flat-\zeta(X)\zeta)
\end{equation*}

for all $X\in\Gamma(TM)$. Moreover, $\dd(\log\escal{\rho,\overline{\rho}})=2\mu\zeta$.
\end{lemma}

\begin{proof}
Let $c:=i\mu-\alpha\in C^\infty(M,\C)$. Then $\nabla^{g,A}_X\rho=c*(X^\flat\wedge\rho)$ for all $X\in\Gamma(TM)$ by Proposition~\ref{prop:complex_Killing}. We first compute the derivative of $u:=\escal{\rho,\overline{\rho}}$. Using $\escal{\rho,*(X^\flat\wedge\overline{\rho})}=-\escal{*(X^\flat\wedge\rho),\overline{\rho}}$ and $c-\overline{c}=2i\mu$ we obtain:
\begin{equation*}
X(u)=2i\mu\escal{*(X^\flat\wedge\rho),\overline{\rho}}.
\end{equation*}

Using $\rho\wedge\overline{\rho}=-iu*\zeta$ we obtain:
\begin{equation*}
\escal{*(X^\flat\wedge\rho),\overline{\rho}}\mathrm{vol}_g =X^\flat\wedge\rho\wedge\overline{\rho}=-iuX^\flat\wedge*\zeta=-iu\zeta(X)\mathrm{vol}_g ,
\end{equation*}

where $\mathrm{vol}_g $ is the Riemannian volume form. Hence $X(u)=2u\mu\zeta(X)$ for all $X\in\Gamma(TM)$, which is equivalent to $\dd(\log u)=2\mu\zeta$.\medskip

Now define the $(0,2)$-tensor field $T$ by $T(X,Y):=\iota_Y(*(*(X^\flat\wedge\rho)\wedge\overline{\rho}))$ for all $X$, $Y\in\Gamma(TM)$. Then:
\begin{align*}
T(X,Y)&=\iota_Y(*(*(X^\flat\wedge\rho)\wedge\overline{\rho}))=*(*(X^\flat\wedge\rho)\wedge\overline{\rho}\wedge Y^\flat)\\
&=-*(Y^\flat\wedge\overline{\rho}\wedge*(X^\flat\wedge\rho))=-*\escal{Y^\flat\wedge\overline{\rho},X^\flat\wedge\rho}\mathrm{vol}_g \\
&=-\escal{X^\flat\wedge\rho,Y^\flat\wedge\overline{\rho}}=\overline{\rho}(X)\rho(Y)-ug(X,Y).
\end{align*}

Using the above relation and:
\begin{align*}
\overline{\rho}(X)\rho(Y)+\rho(X)\overline{\rho}(Y)&=u(g(X,Y)-\zeta(X)\zeta(Y)),\\
\overline{\rho}(X)\rho(Y)-\rho(X)\overline{\rho}(Y)&=(\overline{\rho}\wedge\rho)(X,Y)=iu(*\zeta)(X,Y)
\end{align*}

we obtain:
\begin{align*}
(\nabla^g_X\zeta)Y&=-2\mu\zeta(X)\zeta(Y)+\frac{ic}{u}T(X,Y)-\frac{i\overline{c}}{u}\overline{T(X,Y)}\\
&=\alpha(*\zeta)(X,Y)+\mu(g(X,Y)-\zeta(X)\zeta(Y)).
\end{align*}

Therefore, using $\omega=-*\zeta$, we have $\nabla^g_X\zeta=-\alpha\iota_X\omega+\mu(X^\flat-\zeta(X)\zeta)$ for all $X\in\Gamma(TM)$.
\end{proof}

In particular, from Lemma~\ref{lemma:nabla_zeta} we get $\dd\zeta=-2\alpha\omega$, and since $\mu\zeta$ is exact we obtain:
\begin{equation*}
0=\dd(\mu\zeta)=\dd\mu\wedge\zeta+\mu\dd\zeta=\dd\mu\wedge\zeta-2\alpha\mu\omega.
\end{equation*}

Taking the wedge product with $\zeta$, and using that $\zeta\wedge\omega=-\mathrm{vol}_g $ is nowhere vanishing, we conclude that $\alpha\mu\equiv0$.

\begin{lemma}\label{lemma:nabla_omega}
The two-form $\omega\in\Omega^2(M)$ satisfies:
\begin{equation*}
\nabla^g_X\omega=\alpha X^\flat\wedge\zeta+\mu\iota_X\omega\wedge\zeta
\end{equation*}

for all $X\in\Gamma(TM)$.
\end{lemma}

\begin{proof}
This follows from Lemma~\ref{lemma:nabla_zeta} and $\omega=-*\zeta$ by using $\iota_X\mathrm{vol}_g =*X^\flat$ and $\mathrm{vol}_g =-\zeta\wedge\omega$.
\end{proof}

Define the fundamental $(1,1)$-tensor field $\varphi\in\Gamma(\End(TM))$ by:
\begin{equation*}
\varphi(X):=(\iota_X\omega)^\sharp,
\end{equation*}

which automatically satisfies $g(\varphi(X),Y)=\omega(X,Y)$ and $g(\varphi(X),Y)+g(X,\varphi(Y))=0$ for all $X$, $Y\in\Gamma(TM)$. It is straightforward to check that $\varphi$ also satisfies:
\begin{equation*}
\varphi^2=-\Id+\zeta\otimes\xi,\qquad g(\varphi(X),\varphi(Y))=g(X,Y)-\zeta(X)\zeta(Y),\qquad\zeta\circ\varphi=0.
\end{equation*}

Hence, we recover all the identities in Equation~\eqref{eq:metricontact}. Furthermore, Lemma~\ref{lemma:nabla_omega} can be equivalently written as:
\begin{equation*}
(\nabla^g_X\varphi)Y=\alpha\big(g(X,Y)\xi-\zeta(Y)X\big)+\mu\big(g(\varphi(X),Y)\xi-\zeta(Y)\varphi(X)\big),
\end{equation*}

which is precisely the defining condition for the tuple $(g,\varphi,\xi,\zeta)$ given in Equation~\eqref{eq:trans} to constitute a well-defined trans-Sasakian structure of type $(\alpha,\mu)$ on $M$. Therefore, we have proved the following result.

\begin{prop}
\label{prop:onlyif}
Let $(M,g)$ be a three-dimensional spin-c Riemannian manifold admitting a complex generalized spin-c Killing spinor with complex Killing function $i\mu-\alpha$. Then $(M,g)$ is a trans-Sasakian manifold of type $(\alpha,\mu)$ such that $\mu\zeta$ is exact.
\end{prop}

We will refer to such trans-Sasakian manifolds $(M,g,\varphi,\xi,\zeta)$ whose one-form $\mu\zeta$ is exact as being \emph{$\mu$-exact}.


\subsubsection{Proof of the if direction in Theorem~\ref{thm:ifandonlyif}}


Conversely, assume that $(M,g,\varphi,\xi,\zeta)$ is a three-dimen\-sional $\mu$-exact trans-Sasakian manifold of type $(\alpha,\mu)$, where $\alpha$, $\mu\in C^\infty(M,\R)$. Hence, by assumption, $(g,\varphi,\xi,\zeta)$ satisfies Equations~\eqref{eq:metricontact} and~\eqref{eq:trans}. We will also assume that:
\begin{equation*}
\mu\zeta=\dd f
\end{equation*}

for a function $f\in C^\infty(M,\R)$.

\begin{remark}
The condition $\mu\zeta=\dd f$ implies that $\alpha\mu\equiv0$. Indeed, on a trans-Sasakian manifold of type $(\alpha,\mu)$ we have $\dd\zeta=-2\alpha\omega$ and $\dd\omega=2\mu\zeta\wedge\omega$, where $\omega=g(\varphi-,-)$. Taking the exterior derivative of $\dd f=\mu\zeta$, and then taking the exterior product with $\zeta$ gives $0=-2\alpha\mu\zeta\wedge\omega$. Since $\zeta\wedge\omega$ is nowhere vanishing, we conclude that $\alpha\mu\equiv0$.
\end{remark}

The restriction $J=-\varphi|_{\cD}$ defines an almost complex structure on the distribution $\cD:=\ker(\zeta)\subset TM$, which induces a splitting of its complexification:
\begin{equation*}
\cD_\C:=\cD\otimes\C=\cD^{(1,0)}\oplus\cD^{(0,1)},
\end{equation*}

where $\cD^{(1,0)}$ and $\cD^{(0,1)}$ are the $i$-eigenbundle and $(-i)$-eigenbundle of $J$, respectively. The almost contact metric structure induces a canonical $\mathrm{Spin}^c(3)$-structure on $M$. We will identify its characteristic line bundle with $\mathcal{L}\cong\cD^{(1,0)}$.\medskip

Let $\{e_1,e_2,e_3=\xi\}$ be a local adapted orthonormal frame of $(M,g)$ such that $J(e_1)=-e_2$ and $J(e_2)=e_1$, with corresponding dual coframe $\{e^1,e^2,e^3=\zeta\}$. We consider $(M,g)$ to be equipped with the orientation determined by the volume form:
\begin{equation*}
\mathrm{vol}_g =-e^1\wedge e^2\wedge\zeta.
\end{equation*}

Define the $\mathcal{L}$-valued complex one-form $\rho\in\Omega^1_\C(M,\mathcal{L})$ by:
\begin{equation*}
\rho:=\theta\otimes\mathfrak{l}:=e^f(e^1-ie^2)\otimes(e_1+ie_2),
\end{equation*}

which is nowhere vanishing and isotropic. Moreover, $\rho$ is globally defined since it is invariant under a unitary change of adapted frame.\medskip

Let $\mathcal{P}_\cD\colon TM\to\cD$ be the orthogonal projection onto the distribution $\cD=\ker(\zeta)$, defined explicitly by $\mathcal{P}_\cD(X):=X-\zeta(X)\xi$ for every $X\in TM$. We define a connection $\nabla^\cD$ on the real vector bundle $\cD\to M$ by projecting the Levi-Civita connection $\nabla^g$:
\begin{equation*}
\nabla^\cD_XY:=\mathcal{P}_\cD(\nabla^g_XY)
\end{equation*}

for every $Y\in\Gamma(\cD)$ and $X\in\Gamma(TM)$. Since $(M,g)$ is a trans-Sasakian manifold of type $(\alpha,\mu)$, the restricted endomorphism $J=-\varphi|_\cD$ acts as a parallel complex structure with respect to $\nabla^\cD$, meaning $\nabla^\cD_X(JY)=J\nabla^\cD_XY$. Extending $\nabla^\cD$ complex-linearly to the complexified bundle $\cD_\C$, it preserves the splitting $\cD_\C=\cD^{(1,0)}\oplus\cD^{(0,1)}$ and restricts to a canonical connection on the subbundle $\cD^{(1,0)}$. Define:
\begin{equation*}
\varpi_{12}(X):=g(\nabla^g_Xe_1,e_2)=g(\nabla^\cD_Xe_1,e_2).
\end{equation*}

Then $\nabla^\cD_Xe_1=\varpi_{12}(X)e_2$, $\nabla^\cD_Xe_2=-\varpi_{12}(X)e_1$, and $\nabla^\cD_X\mathfrak{l}=-i\varpi_{12}(X)\mathfrak{l}$. Setting $Y=\xi$ in Equation~\eqref{eq:trans}, and using $\varphi(\xi)=0$ and $\zeta(\xi)=1$, we obtain:
\begin{equation*}
\varphi(\nabla^g_X\xi) = \mu\varphi(X) - \alpha (\zeta(X)\xi-X ).
\end{equation*}

Since $\xi$ has unit length, $\nabla^g_X\xi$ belongs to $\cD := \ker(\zeta)$. Applying $\varphi$ to the identity above and
using $\varphi^2=-\Id+\zeta\otimes\xi$ therefore gives:
\begin{equation*}
\nabla^g_X\xi = -\alpha\varphi(X) +\mu (X-\zeta(X)\xi ) \quad \Longleftrightarrow \quad \nabla^g_X\zeta = -\alpha\iota_X\omega +\mu (X^\flat-\zeta(X)\zeta ).
\end{equation*}

Consequently:
\begin{align*}
\nabla^g_Xe_1&=\varpi_{12}(X)e_2+\big(\alpha g(e_1,\varphi(X))-\mu g(e_1,X)\big)\xi,\\
\nabla^g_Xe_2&=-\varpi_{12}(X)e_1+\big(\alpha g(e_2,\varphi(X))-\mu g(e_2,X)\big)\xi,\\
\nabla^g_X\xi&=-\alpha\varphi(X)+\mu(X-\zeta(X)\xi).
\end{align*}

Then $\nabla^g_X(e^1-ie^2)=i\varpi_{12}(X)(e^1-ie^2)-(\mu+i\alpha)(e^1-ie^2)(X)\zeta$ and:
\begin{equation*}
\nabla^g_X\theta=(\mu\zeta(X)+i\varpi_{12}(X))\theta-(\mu+i\alpha)\theta(X)\zeta,
\end{equation*}

where we have used the assumption $X(f)=\dd f(X)=\mu\zeta(X)$. Consider the connection one-form $A:=\alpha\zeta$. Define:
\begin{equation*}
\nabla^A_X\mathfrak{l}:=\nabla^\cD_X\mathfrak{l}+iA(X)\mathfrak{l}=-i\varpi_{12}(X)\mathfrak{l}+i\alpha\zeta(X)\mathfrak{l}
\end{equation*}

and $\nabla^{g,A}:=\nabla^g\otimes\nabla^A$. Then:
\begin{equation}\label{eq:nabla_rho}
\nabla^{g,A}_X\rho=(\nabla^g_X\theta)\otimes\mathfrak{l}+\theta\otimes(\nabla^A_X\mathfrak{l})=(\mu+i\alpha)(\zeta(X)\rho-\rho(X)\zeta).
\end{equation}

\begin{lemma}
\label{lemma:if}
The $\mathcal{L}$-valued complex one-form $\rho\in\Omega^1_\C(M,\mathcal{L})$ satisfies $\nabla^{g,A}_X\rho=(i\mu-\alpha)*(X^\flat\wedge\rho)$.
\end{lemma}

\begin{proof}
Decompose the vector field $X$ as $X=X_\cD+\zeta(X)\xi$, where $X_\cD\in\Gamma(\cD)$. Then $X^\flat=X_\cD^\flat+\zeta(X)\zeta$. Applying the Hodge star operator, we get:
\begin{equation*}
*(X^\flat\wedge\rho)=*(X_\cD^\flat\wedge\rho)+\zeta(X)*(\zeta\wedge\rho).
\end{equation*}

Using $*e^1=-e^2\wedge\zeta$ and $*e^2=e^1\wedge\zeta$ we see that $*\rho=i\zeta\wedge\rho$. Hence $*(\zeta\wedge\rho)=-i\rho$. Multiplying by $(i\mu-\alpha)\zeta(X)$ we obtain:
\begin{equation*}
(i\mu-\alpha)\zeta(X)*(\zeta\wedge\rho)=(\mu+i\alpha)\zeta(X)\rho.
\end{equation*}

For the other term, we obtain:
\begin{equation*}
*(X_\cD^\flat\wedge\rho)=-\iota_{X_\cD}(*\rho)=i\rho(X_\cD)\zeta=i\rho(X)\zeta
\end{equation*}

since $X_\cD\in\Gamma(\cD)$ and $\rho(\xi)=0$. Multiplying by $i\mu-\alpha$ we obtain:
\begin{equation*}
(i\mu-\alpha)*(X_\cD^\flat\wedge\rho)=-(\mu+i\alpha)\rho(X)\zeta.
\end{equation*}

Combining both terms, we precisely obtain:
\begin{equation*}
(i\mu-\alpha)*(X^\flat\wedge\rho)=(\mu+i\alpha)(\zeta(X)\rho-\rho(X)\zeta),
\end{equation*}

which matches the expression for $\nabla^{g,A}_X\rho$ in Equation~\eqref{eq:nabla_rho}, thereby concluding the proof.
\end{proof}

Proposition~\ref{prop:onlyif} and Lemma~\ref{lemma:if} together prove Theorem~\ref{thm:ifandonlyif}.


\subsection{The counterexample}
\label{sec:counterexample}


Let us consider the following oriented and connected three-dimensional manifold:
\begin{equation*}
M:=\{(x,y,z)\in\R^3\mid x>0\text{ or }y<0\}\subset\R^3.
\end{equation*}

Hence, $M$ is the open \emph{L-shaped} domain that results after removing the quadrant $\left\{x\leq 0 , y\geq 0\right\}$. Choose non-constant functions $F(x)\in C^\infty((-\infty,0))$ and $G(y)\in C^\infty((0,\infty))$ that can be extended by zero to smooth functions on $\R$. Standard choices would be:
\begin{equation}
\label{eq:choiceFG}
F(x) = \begin{cases} 
k_F(x) e^{\frac{1}{x}} & \text{if } x < 0, \\ 
0 & \text{if } x \geq 0, 
\end{cases}\, \qquad G(y) = \begin{cases} 
k_G(y)e^{-\frac{1}{y}} & \text{if } y > 0, \\ 
0 & \text{if } y \le 0 
\end{cases}
\end{equation}

for smooth functions $k_F, k_G \colon \mathbb{R}\to \mathbb{R}$ with discrete zero sets. Define:
\begin{equation*}
\alpha(x,y,z) :=-\tfrac{1}{2}F'(x),\qquad\mu(x,y,z) :=G'(y).
\end{equation*}

Since $F(x)$ and $G(y)$ are both non-constant, neither $F'(x)$ nor $G'(y)$ is identically zero, and thus neither $\alpha$ nor $\mu$ is identically zero. Nevertheless, they satisfy $\alpha\mu\equiv0$. Indeed, since $\mathrm{supp}(\phi')\subseteq\mathrm{supp}(\phi)$ for any real-valued smooth function $\phi$, we have that $F'(x)\neq0$ implies $x<0$, and $G'(y)\neq0$ implies $y>0$. Now suppose that $\alpha(p)\neq0$ at some point $p=(x,y,z)\in M$. Then $x<0$, and since $p\in M$, we must have $y<0$. However, this implies that $G'(y)=0$ since $G'$ is supported in $(0,\infty)$. Hence $\mu(p)=0$. Similarly, if $\mu(p)\neq0$, then $\alpha(p)=0$. Hence, $\alpha\mu\equiv0$ on $M$.\medskip

Define:
\begin{equation*}
e^1:=e^{G(y)}\dd x,\qquad e^2:=e^{G(y)}\dd z,\qquad \zeta:=\dd y+F(x)\dd z
\end{equation*}

and set:
\begin{equation*}
g:=e^1\otimes e^1+e^2\otimes e^2+\zeta\otimes\zeta.
\end{equation*}

Then $\{e^1,e^2,\zeta\}$ is an orthonormal coframe of $(M,g)$ by definition. Choose the orientation:
\begin{equation*}
\mathrm{vol}_g :=-e^1\wedge e^2\wedge\zeta=e^{2G(y)}\dd x\wedge\dd y\wedge\dd z
\end{equation*}

and set $\omega:=e^1\wedge e^2=-*\zeta$. We then have $\dd\zeta=F'(x)\dd x\wedge\dd z=F'(x)e^{-2G(y)}e^1\wedge e^2$. In the region where $F'(x)\neq0$ we have $x<0$, thus in that region we have $y<0$, which implies $G(y)=0$. Consequently:
\begin{equation*}
\dd\zeta=F'(x)e^1\wedge e^2=-2\alpha\omega.
\end{equation*}

We also have $\dd\omega=2\mu\zeta\wedge\omega$. Take the trivial complex line bundle $\mathcal{L}=M\times\C$ and define:
\begin{equation*}
\rho:=\theta\otimes\mathfrak{l}:=e^{G(y)}(e^1-ie^2)\otimes\mathfrak{l},
\end{equation*}

where $\mathfrak{l}$ is a global unit section of $\mathcal{L}$. The one-form $\rho$ is nowhere vanishing, isotropic, and satisfies $\dd(\log\escal{\rho,\overline{\rho}})=2\mu\zeta$, where we have used that $F(x)G'(y)=0$. We define the real connection one-form $A:=-F'(x)\zeta= 2\alpha\zeta$ and set:
\begin{equation*}
\nabla^A \mathfrak{l} := iA \otimes \mathfrak{l} = 2i\alpha\zeta \otimes \mathfrak{l}.
\end{equation*}

Let us now compute $\nabla^{g,A}_X\rho$. First, let $\{e_1,e_2,e_3:=\xi\}$ be the orthonormal frame dual to the coframe $\{e^1,e^2,e^3:=\zeta\}$ and recall that $\dd e^1=-\mu e^1\wedge\zeta$, $\dd e^2=-\mu e^2\wedge\zeta$, and $\dd\zeta=-2\alpha e^1\wedge e^2$. Define the Levi-Civita connection one-forms by $\varpi_{ij}(X):=g(\nabla^g_Xe_i,e_j)$, which satisfy $\varpi_{ij}=-\varpi_{ji}$ and Cartan's equations:
\begin{equation*}
\dd e^i=\sum_{j=1}^3\varpi_{ij}\wedge e^j,\qquad i=1,2,3.
\end{equation*}

Hence, $\varpi_{12}=-\alpha\zeta$, $\varpi_{13}=-\mu e^1-\alpha e^2$, and $\varpi_{23}=\alpha e^1-\mu e^2$. Using these we compute:
\begin{align*}
\nabla^g_Xe^1&=\sum_{j=1}^3\varpi_{1j}(X)e^j=-\alpha\zeta(X)e^2-(\mu e^1(X)+\alpha e^2(X))\zeta,\\
\nabla^g_Xe^2&=\sum_{j=1}^3\varpi_{2j}(X)e^j=\alpha\zeta(X)e^1+(\alpha e^1(X)-\mu e^2(X))\zeta.
\end{align*}

Using these and $X(e^{G(y)})=e^{G(y)}\mu\zeta(X)$ we obtain:
\begin{equation*}
\nabla^g_X\theta=(\mu-i\alpha)\zeta(X)\theta-(\mu+i\alpha)\theta(X)\zeta.
\end{equation*}

Therefore:
\begin{equation*}
\nabla^{g,A}_X\rho=(\nabla^g_X\theta)\otimes\mathfrak{l}+\theta\otimes(\nabla^A_X\mathfrak{l})=(\mu+i\alpha)(\zeta(X)\rho-\rho(X)\zeta).
\end{equation*}

Using $*e^1=-e^2\wedge\zeta$ and $*e^2=e^1\wedge\zeta$ we compute $*(X^\flat\wedge\theta)=-i\zeta(X)\theta+i\theta(X)\zeta$. Therefore:
\begin{equation*}
(i\mu-\alpha)*(X^\flat\wedge\rho)=(\mu+i\alpha)(\zeta(X)\rho-\rho(X)\zeta)=\nabla^{g,A}_X\rho.
\end{equation*}

Using Proposition~\ref{prop:complex_Killing}, we then conclude that $(M,g)$ admits a complex generalized spin-c Killing spinor $\eta\in\Gamma(S)$ with complex Killing function $i\mu-\alpha$, where neither $\alpha$ nor $\mu$ is identically zero. This gives a counterexample to the conjecture stated in~\cite[Rmk.\ 3.5]{Grosse_Nakad_2015}. Here $M=M_{\alpha,0}\cup M_{0,\mu}$, where $M_{\alpha,0}:=\{\mu=0\}$ and $M_{0,\mu}:=\{\alpha=0\}$ are closed subsets of $M$ where $\eta|_{\operatorname{int}(M_{\alpha,0})}$ is a generalized spin-c Killing spinor with real Killing function $-\alpha$ and $\eta|_{\operatorname{int}(M_{0,\mu})}$ is a generalized spin-c Killing spinor with imaginary Killing function $i\mu$.\medskip

The $\mu$-exactness condition of Theorem~\ref{thm:ifandonlyif} can be checked explicitly. We have:
\begin{equation*}
\mu\zeta = G'(y)(\dd y + F(x)\dd z) = G'(y)\dd y + G'(y)F(x)\dd z
\end{equation*}

and thus wherever $G'(y) \neq 0$, we have $y > 0$, which implies $x > 0$ and therefore $F(x) = 0$. Thus, the cross-term $G'(y)F(x)$ vanishes everywhere on $M$, yielding $\mu\zeta = G'(y)\dd y = \dd(G(y))$ globally. If we choose $F$ and $G$ as in Equation~\eqref{eq:choiceFG} with $k_F\equiv k_G\equiv1$, then:
\begin{equation*}
M_{\alpha,0}  = M\setminus \left\{ x > 0 \, , \, y > 0 \right\},\qquad M_{0,\mu}  = M\setminus \left\{ x < 0 \, , \, y < 0 \right\}
\end{equation*}

and, consequently:
\begin{equation*}
M_{\alpha,0}\cap M_{0,\mu} = \left\{ x\geq 0 \, , \, y\leq 0 \right\}\cap M.
\end{equation*}

That is, the solution \emph{glues} a real generalized spin-c Killing spinor in the open quadrant $\left\{ x < 0 \, , \, y < 0 \right\}$ to an imaginary generalized spin-c Killing spinor in the open quadrant $\left\{ x > 0 \, , \, y > 0 \right\}$ via a parallel spinor in the closed quadrant $\left\{ x\geq 0 \, , \, y\leq 0 \right\}\cap M$.


\subsection{Completeness}
\label{sec:completeness}


The construction of the previous counterexample, which is clearly geodesically incomplete and has an incomplete characteristic vector field $\xi$, raises the question of whether a complete example exists. In this section, we prove that this is not possible. More generally, we prove Theorem~\ref{thm:completeReebnogo} in the introduction.

\begin{proof}
Since $(g,\varphi,\xi,\zeta)$ is trans-Sasakian, it satisfies the following equations:
\begin{equation*}
\dd\zeta=-2\alpha\omega, \qquad \dd\omega=2\mu\zeta\wedge\omega.
\end{equation*}

By assumption, $\mu\zeta$ is closed. Hence:
\begin{equation*}
0 =\dd(\mu\zeta)  =\dd\mu\wedge\zeta+\mu\dd\zeta = \dd\mu\wedge\zeta-2\alpha\mu\omega.
\end{equation*}

Taking the exterior product with $\zeta$, and using that $\zeta\wedge\omega$ is nowhere vanishing, we obtain $\alpha\mu\equiv0$, and therefore we conclude that $\dd\mu\wedge\zeta=0$. Evaluating at $\xi$, we obtain:
\begin{equation}
\label{eq:dmuReeb}
\dd\mu=\xi(\mu)\zeta,
\end{equation}

and consequently:
\begin{equation*}
\label{eq:horizontal-mu}
X(\mu)=0
\end{equation*}

for every $X\in\Gamma(\cD)$, where $\cD:=\ker(\zeta)$. On the other hand:
\begin{equation*}
0=\dd^2\zeta = -2\dd\alpha\wedge\omega-2\alpha\dd\omega = - 2 (\xi(\alpha)+2\alpha\mu ) \zeta\wedge\omega,
\end{equation*}

where we have used that $\dd\alpha\wedge\omega = \xi(\alpha)\zeta\wedge\omega$. Consequently:
\begin{equation*}
\xi(\alpha)+2\alpha\mu\equiv0.
\end{equation*}

Since $\alpha\mu\equiv0$, this equation reduces to $\xi(\alpha)=0$. Since we are assuming that $\xi$ is complete, it generates a globally defined flow:
\begin{equation*}
\Phi\colon\R\times M\longrightarrow M, \qquad (t,p)\longmapsto\Phi_t(p).   
\end{equation*}

Furthermore, we have:
\begin{equation*}
\mathscr{L}_\xi\zeta = \dd (\zeta(\xi))+\iota_\xi\dd\zeta = \iota_\xi(-2\alpha\omega) = 0,
\end{equation*}

and therefore the flow of the characteristic vector field $\xi$ preserves both $\zeta$ and its kernel:
\begin{equation*}
\label{eq:Reeb-preserves-D}
\Phi_t^*\zeta=\zeta, \qquad (\Phi_t)_*\cD=\cD
\end{equation*}

for every $t\in\R$. Suppose that $\alpha$ is not everywhere vanishing, and choose a point $p\in M$ such that $\alpha(p)\neq 0$. Let $U$ be the connected component of the open set $\{\alpha\neq0\}$ containing $p$. Restricting to $U$, we have:
\begin{equation*}
(\zeta\wedge\dd\zeta) \vert_U = -2\alpha\,\zeta\wedge\omega\neq 0.    
\end{equation*} 

Hence, the restriction $\cD|_U$ of $\cD$ to $U$ is a contact distribution and, in particular, its sheaf of local sections together with their Lie brackets generate the sheaf of local sections of $TU$, that is:
\begin{equation*}
\cD|_U+[\cD|_U,\cD|_U]=TU.
\end{equation*}

By a classical result of Chow and Rashevskii, it then follows that any two points of $U$ can be joined by a piecewise smooth curve whose tangent vector belongs to $\cD$.\medskip

Let $\mathcal{P}_{\cD}$ denote the pseudogroup generated by the local flows of local sections of $\cD \subset TM$. By definition, an element $\Psi\in\mathcal{P}_{\cD}$ is a finite composition of local flows of such vector fields, that is:
\begin{equation*}
\Psi = \Phi^{X_k}_{t_k}\circ\cdots\circ\Phi^{X_1}_{t_1},
\end{equation*}

where each $X_j$ is a local section of $\cD$ and each $\Phi^{X_j}_{t_j}$ is a local flow map. Every such element $\Psi \in \mathcal{P}_{\cD}$ is a diffeomorphism between open subsets of $M$, since each local flow map is a diffeomorphism onto its image. We define the orbit of $p$ with respect to $\mathcal{P}_{\cD}$ by:
\begin{equation*}
\mathcal{O}_{\cD}(p) := \left\{ \Psi(p)\mid\Psi\in\mathcal{P}_{\cD} \right\}.
\end{equation*}

Implicit in the previous definition of the orbit is the fact that we only consider the elements $\Psi\in \mathcal{P}_{\cD}$ that contain $p$ in their domain, which we denote by $\operatorname{Dom}(\Psi)$. Equivalently, $\mathcal{O}_{\cD}(p)$ is the set of points that can be joined to $p$ by a piecewise smooth curve tangent to $\cD$. By the classical result of Chow and Rashevskii applied to the connected manifold $U$ and the distribution $\cD|_U$, any two points of $U$ can be joined by a piecewise smooth curve tangent to $\cD|_U$. Hence:
\begin{equation}
\label{eq:Uaccessibilityorbit}
U\subseteq\mathcal{O}_{\cD}(p).
\end{equation}

We claim that $\mathcal{O}_{\cD}(p)$ is open in $M$. Let $q\in\mathcal{O}_{\cD}(p)$. By definition, there exists $\Psi\in\mathcal{P}_{\cD}$ such that:
\begin{equation*}
p\in\operatorname{Dom}(\Psi), \qquad \Psi(p)=q.
\end{equation*}

Since $\operatorname{Dom}(\Psi)$ and $U$ are open neighborhoods of $p$, we may choose an open neighborhood $V$ of $p$ such that $V\subseteq U\cap\operatorname{Dom}(\Psi)$. The set $\Psi(V)$ is then an open neighborhood of $q$, because $\Psi$ is a local diffeomorphism. On the other hand, every $m\in V$ belongs to $\mathcal{O}_{\cD}(p)$ by Equation~\eqref{eq:Uaccessibilityorbit}, and applying $\Psi$ shows that $\Psi(m)$ is in the orbit of $m$. Hence, by composition it follows that $\Psi(m)$ is in the orbit of $p$, that is $\Psi(m) \in \cO_{\cD}(p)$. Thus $\Psi(V)$ is an open neighborhood of $q$ contained in $\mathcal{O}_{\cD}(p)$, proving that $\mathcal{O}_{\cD}(p)$ is open.\medskip

Equation~\eqref{eq:dmuReeb} now implies that $X(\mu)=0$ for every local section $X$ of $\cD$. Hence $\mu$ is preserved by every local flow map $\Phi^X_t$ with $X$ tangent to $\cD$, and therefore it is also preserved by every element of $\mathcal{P}_{\cD}$. Since $\alpha(p)\neq 0$ and $\alpha\mu\equiv0$, we have $\mu(p)=0$. Consequently $\mu=0$ on $\mathcal{O}_{\cD}(p)$. Since $\mathscr{L}_\xi\zeta=0$, the flow $\Phi_t$ preserves $\cD$. It follows that conjugation by $\Phi_t$ preserves the pseudogroup associated with $\cD$:
\begin{equation*}
\Phi_t\circ\Psi\circ\Phi_{-t} \in \mathcal{P}_{\cD}
\end{equation*}

whenever $\Psi\in\mathcal{P}_{\cD}$ and the composition is defined. Consequently, the global flow $\Phi_t$ of $\xi$ maps $\cD$-orbits to $\cD$-orbits. More precisely:
\begin{equation}
\label{eq:Reebmapsorbits}
\Phi_t (\mathcal{O}_{\cD}(m) ) = \mathcal{O}_{\cD} (\Phi_t(m) )
\end{equation}

for every $m\in M$ and every $t\in\R$. On the other hand, since $\xi(\alpha)=0$, we have $\alpha(\Phi_t(p))=\alpha(p)\neq 0$ for every $t\in\R$. Therefore, the curve:
\begin{equation*}
s\mapsto\Phi_s(p), \qquad s\in[\min\{0,t\},\max\{0,t\}]
\end{equation*}

is contained in the subset $\{\alpha\neq 0\}\subset M$. Since it starts at $p$, it is contained in its connected component $U$. Hence:
\begin{equation*}
\Phi_t(p)\in U\subseteq\mathcal{O}_{\cD}(p)
\end{equation*}

and therefore:
\begin{equation*}
\mathcal{O}_{\cD}(\Phi_t(p))=\mathcal{O}_{\cD}(p).
\end{equation*}

Combining this with Equation~\eqref{eq:Reebmapsorbits}, we obtain:
\begin{equation}
\label{eq:orbitReebinvariant}
\Phi_t\bigl(\mathcal{O}_{\cD}(p)\bigr)=\mathcal{O}_{\cD}(p)
\end{equation}

for every $t\in\R$.\medskip

Consider now the complement:
\begin{equation*}
C:=M\setminus\mathcal{O}_{\cD}(p).
\end{equation*}

We observe that the set $C$ is invariant under every local flow of a local section of $\cD$. Indeed, if $q\in C$ and $\Phi^X_s(q)$ belonged to $\mathcal{O}_{\cD}(p)$, then applying the inverse local flow $\Phi^X_{-s}$ would imply that $q\in\mathcal{O}_{\cD}(p)$, which is a contradiction. Equation~\eqref{eq:orbitReebinvariant} similarly implies that $C$ is invariant under the global characteristic vector field flow. We claim that $C$ is open. Let $q\in C$, and choose local sections $X_1,X_2$ of $\cD$ which form a basis of $\cD$ on a neighborhood of $q$. For sufficiently small $(s_1,s_2,t)\in\R^3$, define:
\begin{equation*}
\Theta(s_1,s_2,t) := \Phi_t\circ \Phi^{X_2}_{s_2}\circ \Phi^{X_1}_{s_1}(q).
\end{equation*}

By the flow invariance of $C$, the image of $\Theta$ is contained in $C$. Furthermore:
\begin{equation*}
\dd\Theta_{(0,0,0)}(a,b,c) = a X_1(q) + b X_2(q) + c \xi(q).
\end{equation*}

Since $T_qM=\cD_q\oplus\R\xi(q)$, the linear map $\dd\Theta_{(0,0,0)}\colon\R^3\to T_qM$ is an isomorphism. The inverse function theorem therefore implies that the image of $\Theta$ contains an open neighborhood of $q$. This neighborhood is contained in $C$, and hence $C$ is open. Thus both $\mathcal{O}_{\cD}(p)$ and its complement $C$ are open. Since $M$ is connected and $\mathcal{O}_{\cD}(p)$ is nonempty, we conclude that $C$ is empty and $\mathcal{O}_{\cD}(p)=M$, which immediately implies that $\mu=0$ everywhere on $M$. Hence, we have proven that if $\alpha$ is not everywhere vanishing, then $\mu = 0$. Since the complementary case requires that $\alpha$ is everywhere vanishing, we conclude.
\end{proof}

\begin{remark}
Note that the statement $\mathcal{O}_{\cD}(p)=M$ proved above does not imply in general that $\alpha$ is nowhere vanishing, since the orbit $\mathcal{O}_{\cD}(p)$ need not be equal to $U$. Rather, it implies that $\mu$, which is constant along every local flow tangent to $\cD$ and vanishes at $p$, vanishes identically on $M$.  
\end{remark}

\begin{cor} 
Let $(M,g,\varphi,\xi,\zeta)$ be a connected three-dimensional trans-Sasakian manifold of type $(\alpha,\mu)$ such that the one-form $\mu\zeta$ is closed. If the Riemannian metric $g$ is complete, then either $\alpha$ is identically zero on $M$, or $\mu$ is identically zero on $M$. In particular, a complete connected $\mu$-exact trans-Sasakian three-manifold cannot have both $\alpha\not\equiv0$ and $\mu\not\equiv0$.
\end{cor}

\begin{proof}
For a trans-Sasakian manifold we have:
\begin{equation*}
\nabla^g_X\xi = -\alpha\varphi(X) + \mu\bigl(X-\zeta(X)\xi\bigr).
\end{equation*}

Setting $X=\xi$ and using $\varphi(\xi)=0$ and $\zeta(\xi)=1$, we obtain $\nabla^g_\xi\xi=0$. Since $\xi$ has unit length, its integral curves are unit-speed geodesics. The completeness of $g$ implies that these geodesics are defined on all of $\R$. Hence the vector field $\xi$ is complete, and the conclusion follows from Theorem~\ref{thm:completeReebnogo}.
\end{proof}


\appendix



\section{On Gro\textnormal{\ss}e--Nakad's proof of Lemma 3.4}
\label{Ap:Lemma3.4}


Throughout this appendix we will use the convention, notation, and numbering of~\cite{Grosse_Nakad_2015}. In particular, we will work with the negative Clifford convention and their Killing function definition. We refer to Remark~\ref{rmk:convenciones} for the relation with the conventions in this article.\medskip

In Lemma 3.4, Gro\ss e and Nakad state that if $(M^n,g)$ is a Riemannian spin-c manifold carrying a complex generalized Killing spinor $\psi$ with Killing function $K=a+ib$, then $ab=0$. However, the examples obtained in Section~\ref{sec:2d} (of the current paper) and the next example show that this statement does not hold in dimension $n=2$. A careful inspection of the proof of Lemma 3.4 reveals an inconsistency in the normalization conventions used for the exterior derivative and the wedge product. Nevertheless, the pointwise conclusion $ab=0$ is still true in dimensions $n\geq3$, but it requires a different argument.\medskip

In this appendix, we first verify the explicit counterexample of Example~\ref{ex:spin-counterexample} at the spinor level. We then explain the issue in the proof of Lemma~3.4 and give an alternative proof of the statement.

\begin{example}\label{ex:spin-counterexample-neg}
Let $M:=\{z=x+iy\in \mathbb{C}\mid |z|< \tfrac{1}{2}\}$ and equip it with the Euclidean metric $g_{\mathrm{eucl}}$ and its standard spin structure, whose complex spinor bundle we denote by $\Sigma M$. We set $e_1 := \tfrac{\partial}{\partial x}$ and $e_2 := \tfrac{\partial}{\partial y}$ and consider the unit and $\nabla$-parallel spinor $\varphi_1$ such that $e_1 \cdot e_2 \cdot \varphi_1 = i \varphi_1$ and define $\varphi_2 := e_1 \cdot \varphi_1$. Then $\{ \varphi_1, \varphi_2\}$ is a $\nabla$-parallel orthonormal frame of $\Sigma M$, with:
\begin{equation*}
e_1 \cdot \varphi_1 = \varphi_2,\qquad e_1 \cdot \varphi_2 = -\varphi_1, \qquad e_2 \cdot \varphi_1 = -i\varphi_2,\qquad e_2 \cdot \varphi_2 = -i\varphi_1.
\end{equation*}

Let $\sqrt{z}$ be the principal branch of the square root on $\mathbb{C} \setminus\{x\leq0,\, y=0\}$. We define:
\begin{equation*}
\psi := \sqrt{1-z}\, \varphi_1 + \sqrt{\overline z + i}\, \varphi_2, \qquad a(z) + ib(z):= \frac{1}{2\sqrt{1-z}\sqrt{\overline{z}+i}}
\end{equation*}

and compute the $\nabla$-derivative of $\psi$:
\begin{align*}
    \nabla_{e_1} \psi &= - \frac{1}{2 \sqrt{1-z}} \varphi_1 + \frac{1}{2 \sqrt{\overline z + i}} \varphi_2 = (a(z) + ib(z))\, e_1\cdot \psi,\\
    \nabla_{e_2}\psi &= - \frac{i}{2 \sqrt{1-z}} \varphi_1 - \frac{i}{2 \sqrt{\overline z + i}} \varphi_2 = (a(z) + ib(z))\, e_2\cdot \psi.
\end{align*}

Thus, $\psi$ is a complex generalized Killing spinor on the spin manifold $(M,g_{\mathrm{eucl}})$. Moreover:
\begin{equation*}
a(0) +ib(0) = \frac{1}{2\sqrt{i}}= \frac{\sqrt{i}}{2i} = \frac{1-i}{2\sqrt{2}}
\end{equation*}

and, consequently:
\begin{equation*}
a(0) b(0) = \left(\frac{1}{2\sqrt{2}}\right) \left(\frac{-1}{2\sqrt{2}}\right) = -\frac{1}{8}\neq 0.
\end{equation*}

This example contradicts~\cite[Lemma~3.4]{Grosse_Nakad_2015} in dimension two.
\end{example}

We now explain why the proof of Lemma~3.4 does not hold in dimensions greater than two. First, the formula used in the proof of Lemma~3.2 corresponds to the following normalization of the exterior derivative of a $p$-form:
\begin{align*}
(p+1)\dd\omega_p(e_1,\ldots,e_{p+1})&=\sum_{i=1}^{p+1}(-1)^{i-1}e_i(\omega_p(e_1,\ldots,\hat e_i,\ldots,e_{p+1}))\\
&\quad+\sum_{i<j}(-1)^{i+j}\omega_p([e_i,e_j],e_1,\ldots,\hat e_i,\ldots,\hat e_j,\ldots,e_{p+1}).
\end{align*}

Thus, on $p$-forms, their exterior derivative is related to the standard one by:
\begin{equation*}
\dd_{\mathrm{GN}}=\frac{1}{p+1}\dd_{\mathrm{std}}.
\end{equation*}

If one wants the usual graded Leibniz rule:
\begin{equation*}
\dd(\alpha\wedge\beta)=\dd\alpha\wedge\beta+(-1)^p\alpha\wedge\dd\beta
\end{equation*}

to remain valid with this normalization, then the wedge product must be normalized accordingly. More precisely, for $\alpha\in\Omega^p(M)$ and $\beta\in\Omega^q(M)$ one has:
\begin{equation*}
\alpha\wedge_{\mathrm{GN}}\beta=\frac{p!q!}{(p+q)!}\alpha\wedge_{\mathrm{std}}\beta.
\end{equation*}

With these consistent conventions, the formulas of Lemmas~3.2 and 3.3 remain valid. The problem arises in the proof of Lemma~3.4, where the wedge product is evaluated using the standard unnormalized formula. In particular, the first formula on page 185 should contain an additional factor $1/n$:
\begin{equation*}
(\dd a\wedge_{\mathrm{GN}}\omega_{n-1})(e_1,e_2,\ldots,e_n)=\frac{1}{n}\sum_{j=1}^n(-1)^{j+1}\dd a(e_j)\,\omega_{n-1}(e_1,e_2,\ldots,\hat e_j,\ldots,e_n).
\end{equation*}

Indeed, this factor follows from:
\begin{equation*}
\dd a\wedge_{\mathrm{GN}}\omega_{n-1}=\frac{1!(n-1)!}{n!}\dd a\wedge_{\mathrm{std}}\omega_{n-1}=\frac{1}{n}\dd a\wedge_{\mathrm{std}}\omega_{n-1}.
\end{equation*}

With this additional factor, the computation in the proof of Lemma~3.4 yields:
\begin{equation*}
\escal{\dd a\cdot\psi,\psi}=2nabi\abs{\psi}^2.
\end{equation*}

This is precisely the identity obtained in Lemma~3.1(ii). Hence, instead of producing an additional equation from which one can conclude $ab=0$, the corrected computation merely reproduces Lemma~3.1(ii) and therefore gives no new information about $ab$.\medskip

We now prove that $ab=0$ for $n\geq3$. For convenience, throughout the remainder of this appendix we use the normalized operations $\dd=\dd_{\mathrm{GN}}$ and $\wedge=\wedge_{\mathrm{GN}}$ introduced above. With these operations, the formulas in Lemmas 3.2 and 3.3 remain exactly as stated. Fix a point $x\in M$ and suppose by contradiction that $c:=a(x)b(x)\neq0$. All identities below are evaluated at $x$, and equation numbers henceforth refer to those in~\cite{Grosse_Nakad_2015}.

\begin{lemma}
Let $n\geq5$ odd. Then $ab=0$.
\end{lemma}

\begin{proof}
We have assumed that $c\neq0$. Taking $k=0$ in Equation (7) gives $\omega_2=\frac{1}{2c}\dd b\wedge\omega_1$ and taking $k=0$ in Equation (8) gives $\omega_3=\frac{i}{2c}\dd a\wedge\omega_2$. Finally, taking $k=1$ in Equation (7) gives $\omega_4=\frac{1}{2c}\dd b\wedge\omega_3$. Substituting the first two expressions yields:
\begin{equation*}
\omega_4=\frac{i}{4c^2}\dd b\wedge\dd a\wedge\omega_2=\frac{i}{8c^3}\dd b\wedge\dd a\wedge\dd b\wedge\omega_1=0.
\end{equation*}

Arguing by induction we conclude that $\omega_p=0$ for every $p\geq4$. In particular, $\omega_n=0$. However, for a positively oriented orthonormal frame $\{e_1,\ldots,e_n\}$ of $(M,g)$ we have:
\begin{equation*}
e_1\cdots e_n\cdot\psi=i^{-\frac{n+1}{2}}\psi
\end{equation*}

since the complexified volume element $\omega_\C$ acts as the identity in odd dimensions. Hence:
\begin{equation*}
\omega_n(e_1,\ldots,e_n)=\escal{e_1\cdots e_n\cdot\psi,\psi}=i^{-\frac{n+1}{2}}\abs{\psi}^2\neq0
\end{equation*}

since $\psi$ has no zeros by Lemma 3.1(iv). Therefore, we get a contradiction and conclude that $ab=0$.
\end{proof}

\begin{lemma}
Let $n\geq4$ even. Then $ab=0$.
\end{lemma}

\begin{proof}
The proof is analogous to the case $n\geq5$ odd, but here we use Equation (9) from Lemma 3.3.
\end{proof}

\begin{lemma}
Let $n=3$. Then $ab=0$.
\end{lemma}

\begin{proof}
In this case the degree-four induction argument does not apply, but Equation (7) with $k=0$ gives $\dd b\wedge\omega_1=2c\omega_2$. Let $\{e_1,e_2,e_3\}$ be a positively oriented orthonormal frame of $(M,g)$. In dimension three, the complexified volume element satisfies $e_1\cdot e_2\cdot e_3\cdot\psi=-\psi$. Hence:
\begin{equation*}
e_1\cdot e_2\cdot\psi=e_3\cdot\psi,\qquad e_2\cdot e_3\cdot\psi=e_1\cdot\psi,\qquad e_3\cdot e_1\cdot\psi=e_2\cdot\psi
\end{equation*}

and therefore:
\begin{equation*}
\omega_2(e_1,e_2)=\omega_1(e_3),\qquad \omega_2(e_2,e_3)=\omega_1(e_1),\qquad \omega_2(e_3,e_1)=\omega_1(e_2).
\end{equation*}

Using the above relations, and recalling the normalized convention for the wedge product, we obtain:
\begin{equation*}
(\omega_2\wedge\omega_1)(e_1,e_2,e_3)=\frac{1}{3}(\omega_1(e_1)^2+\omega_1(e_2)^2+\omega_1(e_3)^2).
\end{equation*}

We now prove that:
\begin{equation*}
\sum_{j=1}^3\omega_1(e_j)^2=-\abs{\psi}^4.
\end{equation*}

Indeed, at the point $x$ we may identify the spinor fiber with $\C^2$ and choose the Clifford representation:
\begin{equation*}
e_1\cdot=\begin{pmatrix}
    0&1\\
    -1&0
\end{pmatrix},\qquad e_2\cdot=\begin{pmatrix}
    0&i\\
    i&0
\end{pmatrix},\qquad e_3\cdot=\begin{pmatrix}
    i&0\\
    0&-i
\end{pmatrix},
\end{equation*}

for which $e_j\cdot e_j\cdot=-\Id$ and $e_1\cdot e_2\cdot e_3\cdot =-\Id$. Writing $\psi=\binom{z_1}{z_2}$ and recalling that the Hermitian product $\escal{\cdot,\cdot}$ is linear in the first entry, we obtain:
\begin{equation*}
\omega_1(e_1)=z_2\overline{z}_1-z_1\overline{z}_2,\qquad \omega_1(e_2)=i(z_2\overline{z}_1+z_1\overline{z}_2),\qquad \omega_1(e_3)=i(\abs{z_1}^2-\abs{z_2}^2).
\end{equation*}

Therefore:
\begin{equation*}
\sum_{j=1}^3\omega_1(e_j)^2=-4\abs{z_1}^2\abs{z_2}^2-(\abs{z_1}^2-\abs{z_2}^2)^2=-(\abs{z_1}^2+\abs{z_2}^2)^2=-\abs{\psi}^4.
\end{equation*}

Consequently:
\begin{equation*}
(\omega_2\wedge\omega_1)(e_1,e_2,e_3)=-\frac{1}{3}\abs{\psi}^4\neq0,
\end{equation*}

since $\psi$ has no zeros by Lemma 3.1(iv), thus $\omega_2\wedge\omega_1\neq0$. Finally, wedging $\dd b\wedge\omega_1=2c\omega_2$ with $\omega_1$, we obtain:
\begin{equation*}
0=\dd b\wedge\omega_1\wedge\omega_1=2c\omega_2\wedge\omega_1,
\end{equation*}

which implies that $c=0$. Therefore, we get a contradiction and conclude that $ab=0$.
\end{proof}


\bibliographystyle{myamsplain}
\bibliography{biblio}

\end{document}